%% file: Frob-CY.tex
\documentclass[reqno]{amsart}

\usepackage{amscd}
\usepackage{epsfig}
\usepackage{float}
\usepackage{url}
\usepackage[all]{xy}
\usepackage{mathrsfs}
\usepackage{mathpazo}

\newcommand{\N}{{\mathbb N}}
\newcommand{\Z}{{\mathbb Z}}

\newcommand{\C}{{\mathbb C}}
\newcommand{\R}{{\mathbb R}}
\renewcommand{\P}{{\mathbb P}}

\newcommand{\HH}{{\mathcal H}}

\newcommand{\Nu}{\hat{V}}
\newcommand{\OO}{{\mathcal O}}

\newcommand{\VV}{{\mathcal V}}

\newcommand{\ddd}{{\rm d}}
\newcommand{\www}{\widetilde}
\newcommand{\whh}{\widehat}
\newcommand{\oooo}{\overline}

\newcommand{\nmmm}{{\{0\}\times M}}

\newcommand{\immm}{{\{\infty\}\times M}}
\newcommand{\cmmm}{{\C\times M}}
\newcommand{\csmmm}{{\C^*\times M}}

\newcommand{\pmmm}{{\P^1\times M}}

\newcommand{\oomm}{\Omega}
\newcommand{\paa}{\partial}
\newcommand{\zdz}{{z\partial_z}}
\newcommand{\nnn}{\nabla}

\DeclareMathOperator{\Gr}{Gr}

\DeclareMathOperator{\id}{id}

\DeclareMathOperator{\Lie}{Lie}

\DeclareMathOperator{\rank}{rank}
\DeclareMathOperator{\rk}{rk}

\begin{document}

\theoremstyle{plain}
\newtheorem{lemma}{Lemma}[section]
\newtheorem{theorem}[lemma]{Theorem}
\newtheorem{proposition}[lemma]{Proposition}
\newtheorem{corollary}[lemma]{Corollary}
\newtheorem{conjecture}[lemma]{Conjecture}
\newtheorem{conjectures}[lemma]{Conjectures}

\theoremstyle{definition}
\newtheorem{definition}[lemma]{Definition}
\newtheorem{withouttitle}[lemma]{}
\newtheorem{remark}[lemma]{Remark}
\newtheorem{remarks}[lemma]{Remarks}
\newtheorem{example}[lemma]{Example}
\newtheorem{examples}[lemma]{Examples}

\title{Frobenius manifolds, projective special geometry and Hitchin systems}

\author{Claus Hertling\and Luuk Hoevenaars\and Hessel Posthuma}

\address{Claus Hertling\\
Lehrstuhl f\"ur Mathematik VI, Universit\"at Mannheim, Seminargeb\"aude
A 5, 6, 68131 Mannheim, Germany}

\email{hertling\char64 math.uni-mannheim.de}

\address{Luuk Hoevenaars\\
Mathematisch Instituut, Universiteit Utrecht, P.O. BOX 80.010, NL-3508 TA
Utrecht, Nederland}

\email{L.K.Hoevenaars\char64 uu.nl}

\address{Hessel Posthuma\\
Korteweg-de Vries institute for mathematics, University of Amsterdam, P.O. Box 94248,
1090 GE, Amsterdam, The Netherlands}

\email{H.B.Posthuma\char64 math.uva.nl}



\date{19 May 2009}

\maketitle

\begin{abstract}
We consider the construction of Frobenius manifolds associated to projective special geometry and analyse the dependence on choices involved. 
In particular, we prove that the underlying $F$-manifold is canonical.
We then apply this construction to integrable systems of Hitchin type.
\end{abstract}


\setcounter{section}{0}

\section*{Introduction}
\label{intro}
One way of formulating the mirror symmetry conjecture is in terms of 
Frobenius manifolds. On the one hand (the $A$-side) it is well known
that the quantum cohomology product gives rise to a natural Frobenius manifold. 
The other, $B$-side, is constructed from certain variations of Hodge structures. 
In the case of Calabi-Yau threefolds, this $B$-side of the story is perhaps less well-known 
to mathematicians and appears implicitly in \cite{BARA-KONT:1998,BARA:2002,BARA:2001,FERN-PEAR:2004,HERT-MANI:2004}.
 
One of the purposes of this paper is to give an elementary description of 
this $B$-side Frobenius geometry and to specify the dependence on choices 
one has to make. The starting point for this will be a projective special geometry,
or in other words, an abstract variation of Hodge structures of weight 3 of the type considered in \cite{BRYA-GRIF:1983}.
There are two choices involved in the construction: one is a generator for the degree three subspace in the Hodge filtration, i.e., a volume form in the case of Calabi--Yau threefolds. Second, and more importantly, a choice of an {\em opposite filtration}.
Parts of this choice have a natural interpretation in terms of special geometry as 
choices of affine coordinate patches. Other parts lie outside the realm of special geometry.
Remarkably we find that the underlying $F$-manifold, cf. \cite{HERT-MANI:2004}, is independent of all choices.
Also remark that in the geometric context of Calabi--Yau threefolds, the large complex structure limit gives a specific opposite filtration \cite{DELI:1997}.
 
In the second part of the paper, we apply these results to integrable systems of Hitchin type \cite{HITC:1987}.
Namely, first of all, we show that the special geometry on the base cf. \cite{FREE:1999} can be made projective.
Equivalently, the variation of polarized Hodge structures of weight one refines to a variation of Hodge like filtrations of weight three as in \cite{BRYA-GRIF:1983}. 
Underlying this construction is a certain family of cameral curves and a Seiberg--Witten differential that is constructed 
in terms of the natural $\C^*$-action on the total space of the integrable system.
This is closely related to \cite{DIAC-DIJK-DONA-HOFM-PANT:2006, DIAC-DONA-PANT:2007} where in the cases of $ADE$-groups a family of Calabi--Yau threefolds was constructed whose variation of --a priori mixed-- Hodge structures coincides with that of the Hitchin system.

This brings us to the starting data of the first part of the paper, and gives us constructions of the associated Frobenius manifolds.
In this example all choices, abstractly defined in the first part, have a natural interpretation in the elementary geometry of curves.
We believe this Frobenius manifold to be of interest for the following reasons: first of all, as shown in \cite{HAUS-THAD:2003} the Hitchin integrable systems associated to Langlands dual groups are SYZ-mirror to one another. Second, in the geometric transition conjecture for $ADE$-fibered Calabi--Yau threefolds \cite{DIAC-DIJK-DONA-HOFM-PANT:2006}, one of the two --conjecturally equal-- string theories is captured entirely by the underlying special geometry.

\section*{Outline}
A Kuranishi family with base manifold $B_0$ 
of Calaby-Yau threefolds gives rise to a 
variation of polarized Hodge structures of weight 3 on the
primitive part of the the middle cohomology bundle
with some distinguished properties.
Such a VPHS induces on the one hand projective special K\"ahler geometry
on a manifold $B$ with $\dim B=\dim B_0+1$,
on the other hand it induces Frobenius manifold structures on a 
manifold $M$ with $\dim M =2\dim B_0+2$.
Both geometries contain some flat structures and potentials
and both depend on additional choices.

The purpose of the first five sections is to review the 
(well known) constructions, to discuss the dependence on choices
and to give a comparison.
Section \ref{c1} gives definitions, section \ref{c2}
treats Frobenius manifolds, section \ref{c3} shows that the underlying $F$-manifold
structure is independent of choices, section \ref{c5} treats
a part of projective special K\"ahler geometry, 
and section \ref{c6} compares them.

Section \ref{hit} reviews the Hitchin system, section \ref{SWdif} introduces the Seiberg-Witten differential
and section \ref{camVHS} gives the variation of Hodge like filtrations of weight three on the
base of the Hitchin system. Finally, the results of sections $\ref{c1}$--$\ref{c6}$ are applied to the Hitchin system
in section \ref{Frob}.

\section{Some definitions}\label{c1}
\setcounter{equation}{0}

\noindent
In the next five sections $B_0$ will be a small neighborhood
of a base point 0 in a complex manifold of dimension $n$.
When necessary, the size will be decreased, 
so essentially the germ $(B_0,0)$ is considered,
but we will not emphasize this.

\subsection{Variation of polarized Hodge structures (VPHS)}\label{c1.1}

\noindent
A VPHS of weight $w\in\Z$ on $B_0$ consists of data
$(B_0,\mathcal{V},\nnn,\mathcal{V}_\R,S,F^\bullet,w)$.
Here $\mathcal{V}$ is a holomorphic vector bundle with a flat connection
$\nnn$ and a real $\nnn$-flat subbundle $\mathcal{V}_\R$ such that
$\mathcal{V}=\mathcal{V}_\R\otimes \C$, $S$ is a $(-1)^w$-symmetric $\nnn$-flat 
nondegenerate pairing on $\mathcal{V}$ with real values on $\mathcal{V}_\R$,
the decreasing Hodge filtration $F^\bullet$ is a filtration
of holomorphic subbundles with
\begin{eqnarray}\label{1.1}
&& \nnn:\OO(F^p)\to \OO(F^{p-1})\otimes \Omega_{B_0}^1 \quad
\textup{ (Griffiths transversality)},\\ \label{1.2}
&& \mathcal{V}=F^p\oplus \oooo{F^{w+1-p}} \qquad 
\textup{ (Hodge structure)},\quad \textup{ equivalent:}
\\ \label{1.3}
&&  \mathcal{V}=\bigoplus_pH^{p,w-p}
\qquad \textup{ where }H^{p,w-p}:=F^p\cap\overline{F^{w-p}},\\ \label{1.4}
&& S(F^p,F^{w+1-p}) = 0 \qquad
\textup{ (part of the polarization)},\\ \label{1.5}
&& i^{2p-w}S(v,\oooo v)> 0 \quad \textup{ for }\quad 
v\in H^{p,w-p}-\{0\} \\ \nonumber
&&\hspace*{4.5cm} \textup{ (rest of the polarization)}.
\end{eqnarray}
In this chapter the real structure and the conditions 
\eqref{1.2}, \eqref{1.3} and \eqref{1.5} will play no role.
Data $(B_0,\mathcal{V},\nnn,S,F^\bullet,w)$ as above with \eqref{1.1}
and \eqref{1.4} only will be called {\it variation of 
Hodge like filtrations with pairing}.

The connection $\nnn$ induces a Higgs field $C$ on
$\bigoplus_p F^p/F^{p+1}$, 
\begin{eqnarray}\label{1.6}
C =[\nnn]: \OO(F^p/F^{p+1}) \to  \OO(F^{p-1}/F^p)\otimes \Omega^1_{B_0}
\end{eqnarray}
with $C_XC_Y=C_YC_X$ for $X,Y\in T_{B_0}$. From section \ref{c2} on, we consider
only data which satisfy $w=3$ and the two conditions:
\begin{eqnarray}
&&\left. \begin{matrix} 
F^{w+1} = 0,\quad \rank F^w=1,\quad \rank F^{w-1}/ F^w=n,\\
\textup{and thus \ }\rank F^1/F^2=n,\quad 
\rank F^0/F^1=1,\quad F^0=\mathcal{V}.\end{matrix}\right\} \label{1.7}\\
&&\left. \begin{matrix} 
\textup{For any }\lambda_0\in F^w_0-\{0\}\quad\textup{ the map }\\
\quad C_\bullet\lambda_0 : T_0 B_0\to F^{w-1}_0/F^w_0
\quad\textup{ is an isomorphism.}\end{matrix}\right\}  \label{1.8}
\end{eqnarray}
As $B_0$ is small, \eqref{1.8} extends from $0\in B_0$ to all points
in $B_0$.
The conditions \eqref{1.7} and \eqref{1.8} together are called
{\it CY-condition}. 
This condition was discussed in \cite{BRYA-GRIF:1983} and is weaker than
the so-called $H^2$-generating condition considered in \cite[ch. 5]{HERT-MANI:2004}.
%
%

\subsection{Opposite filtrations}\label{c1.2}

\noindent
An {\it opposite filtration} $U_\bullet$ is defined to be an increasing
$\nnn$-flat filtration with
\begin{eqnarray}\label{1.10}
&&\left.\begin{matrix} 
\mathcal{V}= &\bigoplus_p F^p\cap U_p &\textup{ or equivalently }\\
\mathcal{V}=&F^p\oplus U_{p-1},&\end{matrix}\right\} \\ \label{1.11}
&&S(U_p,U_{w-1-p})=0.
\end{eqnarray}
As it is $\nnn$-flat, $U_\bullet$ is determined by $U_\bullet \mathcal{V}_0$
and will be identified with that filtration.
The splitting in \eqref{1.10} is holomorphic,
the one on \eqref{1.3} is only real analytic.
Both are $S$-orthogonal in the sense
\begin{eqnarray}\label{1.12}
S(F^p\cap U_p,F^q\cap U_q)=0=S(H^{p,w-p},H^{q,w-q})\quad
\textup{ if }p+q\neq w.
\end{eqnarray}
$S$ and $U_\bullet$ induce a symmetric and nondegenerate pairing
$g^U$ on $\mathcal{V}$ by 
\begin{eqnarray}\label{1.13}
g^U(a,b) := (-1)^p  S(a,b)\quad \textup{ for }\quad
a\in \OO(F^p\cap U_p),b\in\OO(\mathcal{V}).
\end{eqnarray}
The splitting in \eqref{1.10} is also $g^U$-orthogonal in the sense
of \eqref{1.12}. 

Now the connection $\nnn$ decomposes into $\nnn=\nnn^U+C^U$,
where $\nnn^U$ is a connection on each subbundle $F^p\cap U_p$
and $C^U$ is $\OO_{B_0}$-linear and maps $F^p\cap U_p$ to
$(F^{p-1}\cap U_{p-1})\otimes \Omega_{B_0}^1$.
The flatness of $\nnn$ is equivalent to $\nnn^U$ being flat,
$C^U$ being a Higgs field and the potentiality condition
$\nnn^U(C^U)=0$, more explicitly:
\begin{eqnarray}\label{1.14}
\nnn^U_X(C^U_Y)-\nnn^U_Y(C^U_X) - C^U_{[X,Y]}=0\quad
\textup{ for }X,Y\in T_{B_0}.
\end{eqnarray}
Because of \eqref{1.12} and the $\nnn$-flatness of $S$,
both $S$ and $g^U$ are $\nnn^U$-flat, and $g^U$ satisfies
\begin{eqnarray}\label{1.15}
g^U(C^U_Xa,b)=g^U(a,C^U_Xb)\qquad \textup{ for }\quad
X\in T_{B_0},\ a,b\in \OO(\mathcal{V}).
\end{eqnarray}
That is: $C^U$ is selfadjoint with respect to $g^U$.
Define the endomorphism
$$\Nu^U:\mathcal{V}\to \mathcal{V},\quad \Nu^U:=\sum_p(p-\frac{w}{2})\id_{|F^p\cap U_p}.$$
Then $\Nu^U$ is $\nnn^U$-flat, and $(\Nu^U)^*=-\Nu^U$ where $*$ denotes the adjoint with respect to $g^U$, and 
$[C^U,\Nu^U]=C^U$.

In the case $w=3$, the combination of the CY-condition \eqref{1.7} \& \eqref{1.8}
and of the choice of an opposite filtration $U_\bullet$ leads
to a Frobenius manifold, and the combination of the CY-condition
and of the part $U_1$ of an opposite filtration leads to 
(a part of) projective special K\"ahler geometry.
This will be discussed in the sections \ref{c2} and \ref{c5}.

\subsection{Frobenius manifolds and F-manifolds}\label{c1.3}

\noindent
An {\it F-manifold} $(M,\circ,e,E)$ \cite{HERT-MANI:1999,HERT:2002}
is a complex manifold
of dimension $\geq 1$ with a commutative and associative
multiplication on the holomorphic tangent bundle $T_M$,
a unit field $e\in T_M$ and an Euler field $E\in T_M$ with the
following two properties: $\Lie_E(\circ)=0$ and 
\begin{eqnarray}\label{1.16} 
\Lie_{X\circ Y}(\circ) = X\circ \Lie_Y(\circ) + Y\circ \Lie_Y(\circ).
\end{eqnarray}
\eqref{1.16} implies $\Lie_e(\circ)=0$.

A {\it Frobenius manifold} $(M,\circ,e,E,g)$ \cite{DUBR:1996} is an F-manifold
together with a symmetric nondegenerate $\OO_M$-bilinear pairing
$g$ on $T_M$ with the following properties:
its Levi-Civita connection $\nnn^g$ is flat; there is a potential
$\Phi\in\OO_M$ such that for $\nnn^g$-flat vector fields $X,Y,Z$
\begin{eqnarray}\label{1.17}
g(X\circ Y,Z)=XYZ(\Phi);
\end{eqnarray}
the unit field $e$ is $\nnn^g$-flat; the Euler field $E$ satisfies
$\Lie_E(g)=(2-w)\cdot g$ for some $w\in \C$.

In fact, the potentiality condition and the flatness imply 
\eqref{1.16}, cf.\  \cite{HERT-MANI:1999,HERT:2002}.
They also imply that the metric is multiplication invariant,
\begin{eqnarray}\label{1.18}
g(X\circ Y,Z)=g(X,Y\circ Z) \qquad \textup{ for }\quad
X,Y,Z\in T_M.
\end{eqnarray}
It turns out that $\nnn^g_\bullet E$ is a flat endomorphism
of the tangent bundle.

A Frobenius manifold $M$ with a base point $0\in M$ is called
{\it semihomogeneous} if $w\in\N$ and if there are integers
$0=p_1\leq p_2\leq ... \leq p_{\dim M-1}\leq p_{\dim M}=w$
and flat coordinates $t_i$ centered at $0$ such that
$e=\frac{\partial}{\partial t_1}$ and 
\begin{eqnarray}\label{1.19}
E=\sum_{i=1}^{\dim M} (1-p_i)t_i\frac{\partial}{\partial t_i}.
\end{eqnarray}
Then the numbers $p_i$ are unique, because the numbers $p_i-1$
are the eigenvalues of $\nnn^g_\bullet E$.
The coordinates $t_i$ are called semihomogeneous.

\section{Frobenius manifolds from VPHS of weight 3}\label{c2}
\setcounter{equation}{0}

\noindent
Throughout the whole section \ref{c2} except lemma \ref{c2.3},
a variation $((B_0,0),\mathcal{V},\nnn,S,F^\bullet,w=3)$ of Hodge like
filtrations with pairing (see \ref{c1.1} for this notion) 
of weight 3 with CY-condition \eqref{1.7}\&\eqref{1.8}
and $n=\dim B_0$ is fixed.

Theorem \ref{c2.2} (a) gives a construction of Frobenius manifolds
from it and an additional choice. 
This construction is well known, but usually 
hidden within much richer structures 
\cite{BARA-KONT:1998,BARA:2002,BARA:2001,FERN-PEAR:2004,HERT-MANI:2004}. We will give a proof
which will make the comparison with projective special 
K\"ahler geometry easy.

Theorem \ref{c2.2} (b) shows that the underlying F-manifold
with Euler field is independent of the additional choice,
contrary to the flat structure and the metric.
This result is new. It will be proved in section \ref{c3.6}
(see also remark \ref{c2.3}).

The section starts with lemma \ref{c2.1} which discusses the geometry
of the variation of Hodge like filtrations of weight 3 with CY-condition and
an opposite filtration. It yields coordinates $t_2,...,t_{n+1}$ and a 
prepotential $\Psi\in \OO_{B_0}$. It is complemented by lemma \ref{c2.4}
which constructs the initial data of lemma \ref{c2.1} out of 
coordinates $t_2,...,t_{n+1}$ on $B_0$ and an arbitrary function
$\Psi\in\OO_{B_0}$. It shows that the prepotential $\Psi$ in lemma \ref{c2.1}
is not subject to any hidden conditions.
Finally, proposition \ref{c2.5} discusses the automorphisms of the
F-manifold which underlies the Frobenius manifolds in Theorem \ref{c2.2}.
\begin{lemma}\label{c2.1}
Additionally to the variation of Hodge like filtrations of weight 3
with CY-condition fixed above, choose the following data:
\begin{itemize}{}{}
\item[(1)]
An opposite filtration $U_\bullet$. By section \eqref{1.2} it induces
a flat connection $\nnn^U$ on each subbundle of the splitting 
$\mathcal{V}=\bigoplus_{p=0}^3F^p\cap U_p$; and the pairing $S$ is $\nnn$-flat
and $\nnn^U$-flat.
\item[(2)]
A $\nnn^U$-flat basis $v_1,...,v_{2n+2}$ of $\mathcal{V}$ which is compatible
with the splitting of $\mathcal{V}$ and the pairing $S$, and
\begin{eqnarray}\label{2.1}
\left.\begin{matrix}
v_1\in F^3; &  v_2,...,v_{n+1}\in F^2\cap U_2;\\
v_{n+2},...,v_{2n+1}\in F^1\cap U_1; &  v_{2n+2}\in U_0;
\end{matrix}\right\}  \\
S(v_1,v_{2n+2})=-1;\ S(v_k,v_l)=\delta_{k+n,l}\quad
\textup{ for }2\leq k\leq n+1. \label{2.2}
\end{eqnarray}
\end{itemize}
Let $v_1^0,...,v_{2n+2}^0$ be the $\nnn$-flat (here $\nnn$, not $\nnn^U$)
extension of $v_1(0),...,v_{2n+2}(0)\in \mathcal{V}_0$. Then there are unique
coordinates $t_2,...,t_{n+1}$ on $B_0$ and there is a unique function
$\Psi\in \OO_{B_0}$ which satisfy $\Psi(0)=0$ and 
\begin{eqnarray}\label{2.3}
v_1 = v_1^0 +\sum_{i=2}^{n+1} t_iv_i^0 + 
\sum_{i=2}^{n+1} \frac{\paa\Psi}{\paa t_i}\cdot v^0_{n+i} + 
\left(\sum_{k=2}^{n+1}t_k\frac{\paa}{\paa t_k}-2\right)(\Psi)\cdot v^0_{2n+2}.
\end{eqnarray}
They also satisfy for $i,j,k\in \{2,...,n+1\}$ and $a\in\{n+2,...,2n+1\}$
\begin{eqnarray}\label{2.4}
\nnn_{\frac{\paa}{\paa t_i}} v_1&=&v_i,\\
v_i&=& v_i^0+\sum_j\frac{\paa^2\Psi}{\paa t_i\paa t_j}v_{n+j}^0
+\left(\sum_kt_k\frac{\paa}{\paa t_k}-1\right)
(\frac{\paa\Psi}{\paa t_i})\cdot v_{2n+2}^0,\label{2.5}\\
\nnn_{\frac{\paa}{\paa t_i}} v_j&=&
\sum_k\frac{\paa^3\Psi}{\paa t_i\paa t_j\paa t_k}v_{n+k},\label{2.6}\\
v_a&=&v_a^0+t_{a-n}\cdot v_{2n+2}^0,\label{2.7}\\
\nnn_{\frac{\paa}{\paa t_i}} v_a &=& \delta_{i+n,a}\cdot v_{2n+2},\label{2.8}\\
v_{2n+2}&=& v_{2n+2}^0,\label{2.9}\\ 
\nnn_{\frac{\paa}{\paa t_i}} v_{2n+2}&=&0.\label{2.10}
\end{eqnarray}
The function $\Psi$ and the flat structure on $B_0$ from the coordinates
$t_2,...,t_{n+1}$ depend only on the choice of $U_\bullet$ and $v_1(0)$.
\end{lemma}

\begin{proof}
Here and later the following convention for indices will be used:
\begin{eqnarray}\label{2.11}
i,j,k\in \{2,...,n+1\}, a,b\in\{n+2,...,2n+1\},
\alpha\in \{1,...,2n+2\}.
\end{eqnarray}
Define $p_1=3,p_i=2,p_a=1,p_{2n+2}=0$, 
then $v_\alpha\in \OO(F^{p_\alpha}\cap  U_{p_\alpha})$.
The sections $v_\alpha$ and $v_\alpha^0$ satisfy the following properties:
\begin{eqnarray}\nonumber
v_\alpha^0&\in& \OO(U_{p_\alpha})\quad (\textup{because }
U_{p_\alpha}\textup{ is }\nnn\textup{-flat}),\\ 
\nnn^Uv_\alpha^0&=&-C^Uv_\alpha^0 \in \OO(U_{p_\alpha-1})\otimes \Omega^1_{B_0},
\nonumber\\
v_\alpha&\equiv& v_\alpha^0\mod \OO(U_{p_\alpha-1}),\label{2.12}\\
\nnn v_\alpha &=& C^U v_\alpha \in 
\OO(F^{p_\alpha-1}\cap U_{p_\alpha-1})\otimes \Omega^1_{B_0}.\label{2.13}
\end{eqnarray}
\eqref{2.12} gives \eqref{2.9} and \eqref{2.10}.
There are unique functions $t_i,\kappa_i,\kappa_{2n+2}\in \OO_{B_0}$
such that 
\begin{eqnarray}\label{2.14}
v_1 = v_1^0 +\sum_{i=2}^{n+1} t_iv_i^0 + 
\sum_{i=2}^{n+1} \kappa_i\cdot v^0_{n+i} + \kappa_{2n+2}\cdot v^0_{2n+2}.
\end{eqnarray}
The CY-condition \eqref{1.7} \& \eqref{1.8} 
shows that $t_2,...,t_{n+1}$ are coordinates on $B_0$.
They are centered at 0 because $v_1(0)=v_1^0$.
From now on denote $\paa_i:=\frac{\paa}{\paa t_i}$.
Derivation of \eqref{2.14} gives
\begin{eqnarray}\label{2.15}
\nnn_{\paa_i}v_1 &=& 
v_i^0+\sum_j\paa_i\kappa_{j}\cdot v_{n+j}^0+
\paa_i\kappa_{2n+2}\cdot v_{2n+2}^0\\
&\equiv& v_i^0 \equiv v_i \mod \OO(U_1).\nonumber
\end{eqnarray}
This and \eqref{2.13} show \eqref{2.4}.

Now we will use two times the pairing $S$, first for 
\eqref{2.7} \& \eqref{2.8}, second for the existence of the function $\Psi$.
Equation \eqref{2.13} shows $\nnn_{\paa_i}v_a\in \OO_{B_0}\cdot v_{2n+2}$.
The coefficient $\delta_{i+n,a}$ in \eqref{2.8}
is determined by
\begin{eqnarray*}
0&=& \paa_i(0) = \paa_i S(v_1,v_a) 
= S(\nnn_{\paa_i} v_1,v_a) + S(v_1,\nnn_{\paa_i}v_a)\\
&=& S(v_i,v_a)+S(v_1,\nnn_{\paa_i}v_a) 
= \delta_{i+n,a}+ S(v_1,\nnn_{\paa_i}v_a)
\end{eqnarray*}
and $S(v_1,v_{2n+2})=-1$. This shows \eqref{2.7} \& \eqref{2.8}.
Condition \eqref{2.2} also holds with $v_\alpha$ replaced by $v_\alpha^0$ 
because $S$ is $\nnn$-flat and $v_\alpha^0(0)=v_\alpha(0)$.
This, together with \eqref{2.4} and \eqref{2.15} give 
\begin{eqnarray*}
0 = S(v_i,v_j) = \paa_j \kappa_i\cdot S(v_i^0,v_{n+i}^0)
+ \paa_i \kappa_j\cdot  S(v_{n+j}^0,v_j^0)
=  \paa_j \kappa_i - \paa_i \kappa_j .
\end{eqnarray*}
Therefore there is a unique function $\Psi\in\OO_{B_0}$
with $\Psi(0)=0$ and 
\begin{eqnarray}\label{2.16}
 \partial_i \Psi= \kappa_i.
\end{eqnarray}
Derivation by $\paa_i$ of \eqref{2.15} and \eqref{2.4} gives
\begin{eqnarray}\label{2.17}
\nnn_{\paa_i}v_j &=& 
\sum_k\paa_i\paa_j\paa_k\Psi\cdot v_{n+k}^0 
+\paa_i\paa_j\kappa_{2n+2}\cdot v_{2n+2}^0\\
&\equiv& \sum_k\paa_i\paa_j\paa_k\Psi\cdot v_{n+k}^0
\equiv \sum_k\paa_i\paa_j\paa_k\Psi\cdot v_{n+k}
\mod \OO(U_0).\nonumber
\end{eqnarray}
This and \eqref{2.13} show \eqref{2.6}.
Using \eqref{2.7} this gives
\begin{eqnarray*}
\nnn_{\paa_i}v_j &=& \sum_k\paa_i\paa_j\paa_k\Psi\cdot v_{n+k}\\
&=& \sum_k\paa_i\paa_j\paa_k\Psi\cdot v_{n+k}^0 
+\sum_k\paa_i\paa_j\paa_k\Psi\cdot t_k\cdot v_{2n+2}^0.
\end{eqnarray*}
With \eqref{2.17} it implies
\begin{eqnarray*}
\paa_i\paa_j\kappa_{2n+2} &=& \sum_k t_k\cdot \paa_i\paa_j\paa_k\Psi\\
&=& \left(\sum_kt_k\paa_k\right)\paa_i\paa_j\Psi 
=\paa_i\paa_j\left(\sum_kt_k\paa_k-2\right)\Psi.
\end{eqnarray*}
We can conclude 
\begin{eqnarray}\label{2.18}
\kappa_{2n+2} = \left(\sum_kt_k\paa_k-2\right)\Psi
\end{eqnarray}
because we know
\begin{eqnarray*}
\left(\paa_i\kappa_{2n+2}\right)(0)&=&0 
\qquad\left[\Leftarrow \quad v_i(0)=v_i^0 \textup{ and }\eqref{2.15}\right],\\
\left(\paa_i\left(\sum_kt_k\paa_k-2\right)\Psi\right)(0)&=&
\left(\left(\sum_kt_k\paa_k-1\right)\paa_i\Psi\right)(0)\\
&=&-\kappa_i(0)=0
\qquad\left[\Leftarrow \quad v_1(0)=v_1^0\right],\\
\kappa_{2n+2}(0)&=&0 
\qquad\left[\Leftarrow \quad v_1(0)=v_1^0\right],\\
\left(\left(\sum_kt_k\paa_k-2\right)\Psi\right)(0)&=&-2\Psi(0)=0.
\end{eqnarray*}
Now all equations \eqref{2.3} - \eqref{2.10} are proved.

In order to show that the function $\Psi$ and the flat structure on $B_0$
from $t_2,...,t_{n+1}$ are independent of the choice 
of $v_\alpha$ except $v_1(0)$, we suppose that a second choice 
$\www v^0_\alpha$ is made with $\www v_1(0)=v_1(0)$.
All its data are denoted using a tilde.
The base change from the base $(v_\alpha)$  
to the base $(\www v_\alpha)$ is constant,
because both bases are flat with respect to $\nnn^U$.
Now  $\www v_1 = v_1$ because both are $\nnn^U$-flat extensions
of $\www v_1(0)=v_1(0)$. 
With $S(v_1,v_{2n+2})=-1=S(\www v_1,\www v_{2n+2})$ we obtain also 
$\www v_{2n+2}=v_{2n+2}$.
Suppose $(\www v_i)= (v_i)\cdot A$, where $i\in\{2,...,n+1\}$
and $(\www v_i), (v_i)$ are row vectors, and $A\in \textup{Gl}(n,\C)$.
Then 
$(t_i) = (\www t_i)\cdot A^{tr}$, 
$(\www\paa_i) = (\paa_i)\cdot A$,
$(\www\kappa_i)=(\kappa_i)\cdot A$ 
and thus $\www\Psi = \Psi$.
This shows the desired independencies of the flat structure and $\Psi$.

In fact, rescaling of $v_1(0)\in F^3_0-\{0\}$ leads to a rescaling
of $\Psi$, but it does not affect the flat structure on $B_0$.
That depends only on $U_\bullet$.
\end{proof}

\begin{theorem}\label{c2.2}
A variation $((B_0,0),\mathcal{V},\nnn,S,F^\bullet)$ of Hodge like
filtrations of weight 3 with pairing
and with CY-condition \eqref{1.7} \& \eqref{1.8}
and $n=\dim B_0$ is fixed.
\begin{itemize}
\item[(a)] Any choice of an opposite filtration $U_\bullet$ and a generator
$\lambda\in F^3_0$
leads in a canonical way (described below in the proof)
to a Frobenius manifold 
$M^{U,\lambda}\supset\C\times B_0$ of dimension
$2n+2$.
It is semihomogeneous with integers
$(p_1,...,p_{2n+2})=(0,1,...,1,2,...,2,3)$ ($1$ and $2$ each $n$ times).
\item[(b)] The manifold $M^{U,\lambda}$ is canonically isomorphic to the manifold
$M=\C\times B_2$ which is constructed in section \ref{c3.6}.
All Frobenius manifolds induce the same unit field $e$, Euler field $E$
and multiplication $\circ$ on $M$, so the same F-manifold structure.
But in the general the metrics and flat structures differ.
\end{itemize}
The unit field is $e=\frac{\paa}{\paa t_1}$ if $t_1$ is the coordinate
on $\C$ of a coordinate system which respects the product $M=\C\times B_2$.
The manifold $B_2$ comes equipped with a projection $p_2:B_2\to  B_0$,
the fibers are isomorphic to $\C^{n+1}$ as affine algebraic manifolds.
The Euler field induces a good $\C^*$-action on the fibers
with weights $(1,...,1,2)$.
\end{theorem}

Part (b) will follow from the results in section \ref{c3.6}.

\begin{proof}[Proof of part (a)]
All the data in lemma \ref{c2.1} will be used, and also the convention
\eqref{2.11}. The section $v_1$ is chosen such that $v_1(0)=\lambda$.
Define $M^{U,\lambda}=\C\times B_0\times \C^{n+1}$ with coordinates
$(t_1,...,t_{2n+2})$ and the flat connection
defined by these coordinates. Here $(t_2,...,t_{n+1})$ extend the
coordinates on $B_0$ from lemma \ref{c2.1}. 
Denote $\paa_\alpha=\frac{\paa}{\paa t_\alpha}$ for $\alpha=1,...,2n+2$.
Define a potential
\begin{eqnarray}\label{2.19}
\Phi(t_1,...,t_{2n+2}) := \Psi(t_2,...,t_{n+1}) + 
\frac{1}{2}t_1^2t_{2n+2} + t_1\sum_{i=2}^{n+1}t_it_{n+i}.
\end{eqnarray}
Define a symmetric nondegenerate flat bilinear form $g$ on $TM^{U,\lambda}$ by 
\begin{eqnarray}\label{2.20}
\begin{matrix}
g(\paa_1,\paa_\alpha) = \delta_{\alpha,2n+2},& & 
g(\paa_i,\paa_\alpha)=\delta_{i+n,\alpha},\\
g(\paa_{n+i},\paa_\alpha)=\delta_{i,\alpha}, & &
g(\paa_{2n+2},\paa_\alpha)=\delta_{\alpha 1}.
\end{matrix}
\end{eqnarray}
Then $\Phi$, $g$ and formula \eqref{1.17} give the following 
multiplication $\circ$ on $T_M^{U,\lambda}$.
\begin{eqnarray}\nonumber
\paa_1\circ = \id;\quad 
\paa_i\circ\paa_j =\sum_k \paa_i\paa_j\paa_k\Psi \cdot \paa_{n+k},\quad
\paa_i\circ\paa_{a}=\delta_{i+n,a}\cdot \paa_{2n+2},\\ \label{2.21}
\paa_i\circ \paa_{2n+2}=0,\quad \paa_{a}\circ\paa_{b}=0,\quad 
\paa_{a}\circ \paa_{2n+2}=0,\quad
\paa_{2n+2}\circ \paa_{2n+2}=0.
\end{eqnarray}
So, it respects the grading 
$\bigoplus_{p=0}^3 (\bigoplus_{p_\alpha=p}\OO_M^{U,\lambda}\cdot\paa_\alpha)$ of $T_M^{U,\lambda}$,
and the multiplication coefficients depend at most on $t_2,...,t_{n+1}$. 
We claim that it is commutative and associative.
Commutativity of this multiplication is clear.
The only nontrivial part of the associativity is given by
\begin{eqnarray}\nonumber
g((\paa_i\circ \paa_j)\circ\paa_k,\paa_1) &=& 
g(\paa_i\circ\paa_j,\paa_k) = \Psi_{ijk} = \Psi_{jki} \\ \nonumber
&=& g(\paa_j\circ\paa_k,\paa_i) = 
g(\paa_i\circ ( \paa_j\circ\paa_k),\paa_1).
\end{eqnarray}
Define the Euler field $E$ as 
\begin{eqnarray}\label{2.22}
E = \sum_{\alpha=1}^{2n+2}(1-p_\alpha)\cdot t_\alpha\frac{\paa}{\paa t_\alpha}.
\end{eqnarray}
Then $E\Phi = E\Psi=0$, $\Lie_E\paa_\alpha = (p_\alpha-1)\paa_\alpha$,
and $\Lie_E(\circ)=\circ$ and $\Lie_E(g)=(2-3)g$ follow
immediately. Thus $(M^{U,\lambda},\circ,e,E,g)$ is a semihomogeneous 
Frobenius manifold with semihomogeneous coordinates $(t_1,...,t_{2n+2})$.

In order to show that this Frobenius manifold is independent of the choice 
of $v^0_\alpha$ except $v_1(0)=\lambda$, 
we continue the argument from the end of the proof of lemma \ref{c2.1},
with the same second choice $\www v_\alpha$ and the same matrix $A$.
Define the isomorphism $\www M\to M^{U,\lambda}$ by 
$\www t_1=t_1,$ $\www t_{2n+2}=t_{2n+2}$, 
$(\www t_{n+i}) = (t_{n+i})\cdot A$.
Then $\www g = g$, $\www\Phi = \Phi$, $\www E = E$,
so one obtains the same Frobenius manifold.
\end{proof}

\begin{remark}\label{c2.3}
Different choices in lemma \ref{c2.1} and theorem \ref{c2.2}
lead to different coordinate systems on $B_0$ and on $M$
and different flat structures. The coordinate changes are 
very complicated when $U_1$ is changed. Trying to prove theorem \ref{c2.2} (b)
by controlling these coordinate changes looks hard.

But when $U_1$ is fixed and only $U_0$ and $U_2=(U_0)^{\perp_S}$ are
changed, the coordinate changes are much simpler.
This is addressed in section \ref{c6.2} which shows that all the 
coordinates and functions $\Psi$ for fixed $U_1$ and varying $U_0$ \& $U_2$
have a nice common origin from projective special geometry.
We now sketch the common origin of the Frobenius algebra at the level of tangent spaces:
consider
\[
\Gr\left( F_0^\bullet \right)= F_0^3 \oplus \left( F_0^2 / F_0^3 \right) \oplus \left( F_0^1 / F_0^2 \right) \oplus \left( F_0^0 / F_0^1 \right) 
\]
The pairing $S$ on $\mathcal{V}$ induces a bilinear form on $\Gr\left( F_0^\bullet \right)$ whose symmetrization $g$ (the construction is similar to $g^U$ in \eqref{1.13}) gives the pairing of the Frobenius algebra. In order to define the multiplication we recall the definition of the Higgs field $C$ and the isomorphism $T_0 B_0 \cong F_0^2 / F_0^3$ resulting from the Calabi-Yau condition together with a choice of nonzero $\lambda_0 \in F^3_0$.

We define the following commutative, associative, unital and graded multiplication $\circ$ on $\Gr\left( F_0^\bullet \right)$:
\begin{equation*}
\begin{array}{rclrlll}
\lambda_0                 &\circ &  V &=& V  &   \forall V \in \Gr\left( F_0^\bullet \right) \\
C_X \lambda_0         & \circ & C_Y\lambda_0   &=& C_XC_Y\lambda_0 & \forall X,Y \in  T_0B_0 \\
C_X \lambda_0         & \circ & W   &=& C_X W   &      \forall X\in T_0B_0,\; W \in F_0^1/F_0^2
\end{array}
\end{equation*}
Other multiplications are zero unless they are required for commutativity.
Associativity (and commutativity) of $\circ$ follows immediately from the fact that the Higgs field gives commuting endomorphisms: for instance, associativity follows from
\[
C_X \left( C_YC_Z \lambda_0 \right) = C_Z \left( C_X C_Y \lambda_0 \right)
\]
Together with $g$, this multiplication gives a Frobenius algebra which is to be compared with the one given in \eqref{2.20}, \eqref{2.21}.
A different choice of $\lambda_0$ simply gives a rescaling of the multiplication.

Given the part $U_1$ of an opposite filtration, one finds
\[
\Gr\left( F_0^\bullet \right) \cong F_0^3 \oplus F_0^2/F_0^3 \oplus U_1
\]
and the multiplication $\circ$ and bilinear form $g$ can be transferred to the right hand side.
It is possible to identify this space with the tangent space of a manifold, in the following way.
Consider the holomorphic vector bundle $U_1 \to B_0$. 
Using the line bundle $\rho:F^3 \to B_0$ we can pull back this bundle to $\rho^*U_1 \to F^3$. This gives the isomorphism
\[
T_{((0,c),v)}  \rho^*U_1 \cong F^3_0 \oplus T_0B_0 \oplus U_1 \cong F_0^3 \oplus F_0^2/F_0^3 \oplus U_1
\]
where $c \in F^3_0, v\in U_1$. 
So we can view $\Gr\left( F_0^\bullet \right)$ as a tangent space to (the total space of) $\rho^*U_1$.
It is true that refining $U_1$ to a full opposite filtration $U_\bullet$ allows one to use the bilinear form $g^U$ together with $\circ$ to define a Frobenius manifold structure on $\rho^*U_1$.
However, from these considerations it does not follow that all these manifolds are isomorphic as $F$-manifolds to one and the same $M$.
This is the subject of section \ref{c3}.
\end{remark}
\begin{lemma}\label{c2.4}
Let $(B_0,0)$ be a germ of a manifold with coordinates $t_2,...,t_{n+1}$
centered at 0, i.e. $t_2(0)=...=t_{n+1}(0)$,
and let $\Psi\in \OO_{B_0}$ be an arbitrary function with $\Psi(0)=0$.
Furthermore, let $\mathcal{V}\to B_0$ be a holomorphic vector bundle with two bases 
$v_1,...,v_{2n+2}$ and $v_1^0,...,v_{2n+2}^0$ of sections which are related
by \eqref{2.3}, \eqref{2.5}, \eqref{2.7} and \eqref{2.9}.
\begin{itemize}
\item[(a)] Let $\nnn$ be the unique flat connection on $\mathcal{V}$ with flat sections
$v_1^0,...,v_{2n+2}^0$. Then \eqref{2.4}, \eqref{2.6}, \eqref{2.8}
and \eqref{2.10} hold.
\item[(b)] Define two filtrations $F^\bullet$ and $U_\bullet$ on $\mathcal{V}$ by \eqref{2.1}
and an antisymmetric pairing $S$ by \eqref{2.2} and
\begin{equation}\label{2.23}
\begin{split}
S(v_\alpha,v_\beta)=0\quad\textup{for }
(\alpha,\beta)\notin&\{(1,2n+2),(2n+2,1)\}\\
&\cup\{(i,i+n),(i+n,i)\ |\ 
i=2,...,n+1\}.
\end{split}
\end{equation}
Then $((B_0,0),\mathcal{V},\nnn,S,F^\bullet)$ is a Hodge like filtration with 
pairing of weight $w=3$ and with CY-condition \eqref{1.7} \& \eqref{1.8},
and $U_\bullet$ is an opposite filtration.
\end{itemize}
\end{lemma}

\begin{proof}
Again we use the convention \eqref{2.11} and write 
$\paa_i=\frac{\paa}{\paa t_i}$. 

(a) \eqref{2.4}, \eqref{2.8} and \eqref{2.10} are obvious,
\eqref{2.6} follows from
\begin{eqnarray*}
\nnn_{\paa_i}v_j &=&\sum_k\paa_i\paa_j\paa_k\Psi\cdot v_{n+k}^0 
+ (\sum_kt_k\paa_k)\paa_i\paa_j\Psi\cdot v_{2n+2}^0\\
&=& \sum_k \paa_i\paa_j\paa_k\Psi\cdot
(v_{n+k}^0+t_k\cdot v_{2n+2}^0) 
=\sum_k\paa_i\paa_j\paa_k\Psi\cdot v_{n+k}.
\end{eqnarray*}

(b) $\nnn$ and $F^\bullet$ satisfy Griffiths transversality \eqref{1.1}
because of \eqref{2.4}, \eqref{2.6}, \eqref{2.8} and \eqref{2.10}.
For the same reason $U_\bullet$ is $\nnn$-flat.
By definition $S$ satisfies \eqref{1.4} and \eqref{1.12}.

The definition of $S$ in \eqref{2.2} and \eqref{2.23} and 
the formulas \eqref{2.3}, \eqref{2.5}, \eqref{2.7} and \eqref{2.9}
show \eqref{2.2} and \eqref{2.23} for $v_\alpha^0$ instead of
$v_\alpha$. Therefore $S$ is $\nnn$-flat.

Finally, the CY-conditions also hold, \eqref{1.7} is built-in,
\eqref{1.8} follows from \eqref{2.4}
\end{proof}

\begin{proposition}\label{c2.5}
Consider the same data as in theorem \ref{c2.2} (a) and the 
Frobenius manifold constructed there in its proof, 
including the additional choice of 
coordinates $(t_1,...,t_{2n+2})$. Consider the group
\begin{eqnarray*}
\textup{Aut}(M,B_0,\circ, e, E)
&:=&\{\varphi:(M,0)\to (M,0)\textup{ biholomorphic}\ |\\
&&\varphi_{|B_0}=\id_{|B_0},\ 
\varphi\textup{ respects multiplication,}\\
&&\textup{ unit field and Euler field}\}
\end{eqnarray*}
of automorphisms of the underlying F-manifold which fix the 
submanifold $B_0$. 
We use the same convention for the indices as in the proof of 
theorem \ref{c2.1} (a):
$$i,j,k\in\{2,...,n+1\}, a,b\in \{n+2,...,2n+1\},
\alpha\in\{1,...,2n+2\}.$$

(a) For any automorphism $\varphi\in \textup{Aut}(M,B_0,\circ,e,E)$
there exist $\beta\in\C^*$ and $\gamma_{ab}\in \OO_{B_0}$ with
\begin{eqnarray}\nonumber 
\varphi_1&=&t_1,\ \varphi_i=t_i,\ \varphi_a=\beta\cdot t_a,\\
\varphi_{2n+2}&=& \beta\cdot t_{2n+2}
+\sum_{a,b}\gamma_{ab}(t_2,...,t_{n+1})\cdot t_at_b
\label{2.24}
\end{eqnarray}
and 
\begin{eqnarray}\label{2.25}
0=(\paa_i\circ \paa_j)(\varphi_{2n+2})
\left[ = (\sum_{k}\paa_i \paa_j \paa_k \Psi\paa_{k+n})(\varphi_{2n+2})\right].
\end{eqnarray}

(b) In the case when all $\paa_i \paa_j \paa_k\Psi=0$ then \eqref{2.25} is empty,
and $\beta$ and the $\gamma_{ab}$ can be chosen freely.

(c) If some $\paa_i \paa_j \paa_k\Psi\neq 0$ then $\beta=1$, but the $\gamma_{ab}$
are only subject to condition \eqref{2.25}.
%
\end{proposition}

\begin{proof}
Consider an automorphism $\varphi:(M,0)\to (M,0)$. 
The three conditions $\varphi_*(e)=e$, $\varphi_*(E)=E$
and $\varphi_{|B_0}=\id_{|B_0}$ are equivalent to the following:
\begin{eqnarray}\nonumber
\varphi_1&=&t_1,\quad \varphi_i=t_i,\\
\varphi_a&=&\sum_b \beta_{ab}(t_2,...,t_{n+1})\cdot t_b\nonumber\\
&&\hspace*{1cm}
\textup{ for some }\beta_{ab}\in \OO_{B_0}\textup{ with }
\det(\beta_{ab})\in \OO^*_{B_0},\nonumber\\
\varphi_{2n+2}&=&\sum_{a,b}\gamma_{ab}(t_2,...,t_{n+1})\cdot t_at_b
+ \beta(t_2,...,t_{n+1})\cdot t_{2n+2}\nonumber\\
&&\hspace*{1cm}
\textup{ for some }\gamma_{ab}\in \OO_{B_0},\beta\in\OO^*_{B_0}.
\label{2.26}
\end{eqnarray}
Then the coordinate vector fields and their images under $\varphi_*$
satisfy
\begin{eqnarray}\nonumber
\varphi_*(\paa_1)&=&\paa_1,\quad 
\varphi_*(\paa_{2n+2})=\paa_{2n+2}(\varphi_{2n+2})\cdot\paa_{2n+2}
=\beta\cdot \paa_{2n+2},\\
\varphi_*(\paa_a)&=&\sum_b\beta_{ba}\cdot \paa_b 
+ \paa_a(\varphi_{2n+2})\cdot \paa_{2n+2},\nonumber\\
\varphi_*(\paa_i)&=& \paa_i + \sum_a\paa_i(\varphi_a)\cdot \paa_a
+\paa_i(\varphi_{2n+2})\cdot \paa_{2n+2}.
\label{2.27}
\end{eqnarray}
The additional condition that $\varphi$ respects the multiplication
reduces in view of \eqref{2.21} and \eqref{2.27} to the conditions
\begin{eqnarray}\label{2.28}
\varphi_*(\paa_i)\circ\varphi_*(\paa_a)=\varphi_*(\paa_i\circ\paa_a)
\ \textup{ and }\
\varphi_*(\paa_i)\circ\varphi_*(\paa_j)=\varphi_*(\paa_i\circ\paa_j).
\end{eqnarray}
The first one is equivalent to
\begin{eqnarray*}
&&\delta_{i+n,a}\cdot\beta\cdot\paa_{2n+2}
= \delta_{i+n,a}\cdot\varphi_*(\delta_{2n+2}) 
= \varphi_*(\paa_i\circ \paa_a)\\
&=& \varphi_*(\paa_i)\circ\varphi_*(\paa_a) 
=\sum_b\beta_{ba}\cdot\delta_{i+n,b}\cdot \paa_{2n+2}
=\beta_{i+n,a}\cdot \paa_{2n+2}.
\end{eqnarray*}
This is equivalent to $\beta_{ab}=\delta_{ab}\cdot \beta$ and to
\begin{eqnarray}\label{2.29}
\varphi_a=\beta\cdot t_a.
\end{eqnarray}
Taking this into account, the second equation in \eqref{2.28} becomes
\begin{eqnarray*}
&& \sum_k\paa_i \paa_j \paa_k\Psi\cdot \left(\beta\cdot \paa_{k+n}+ 
\paa_{k+n}(\varphi_{2n+2})\cdot\paa_{2n+2}\right) \\
&=& \sum_k\paa_i \paa_j \paa_k\Psi\cdot\varphi_*(\paa_{k+n})
= \varphi_*(\paa_i\circ\paa_j) 
= \varphi_*(\paa_i)\circ\varphi_*(\paa_j) \\
&=& \sum_k\paa_i \paa_j \paa_k\Psi\cdot \paa_{k+n} + 
\sum_a\paa_i\circ \paa_j(\beta)\cdot t_a\paa_a 
+ \sum_a \paa_i(\beta)\cdot t_a\paa_a\circ \paa_j \\
&=& \sum_k\paa_i \paa_j \paa_k\Psi\cdot\paa_{k+n} + 
\left[t_{i+n}\cdot\paa_j(\beta) + t_{j+n}\cdot\paa_i(\beta)\right]
\cdot \paa_{2n+2}.
\end{eqnarray*}
This is equivalent to 
\begin{eqnarray}\label{2.30}
\paa_i \paa_j \paa_k\Psi\cdot \beta =\paa_i \paa_j \paa_k\Psi\qquad \textup{for all }i,j,k
\end{eqnarray}
and 
\begin{eqnarray}\label{2.31}
(\paa_i\circ\paa_j)(\varphi_{2n+2}) 
&&\left[ =\sum_k \paa_i \paa_j \paa_k\Psi\paa_{k+n}(\varphi_{2n+2})\right]
\nonumber \\
&&=t_{i+n}\cdot\paa_j(\beta)+t_{j+n}\cdot\paa_i(\beta).
\end{eqnarray}

{\bf 1st case, all $\paa_i \paa_j \paa_k\Psi=0$:}
then \eqref{2.30} is empty and \eqref{2.31} becomes
$\paa_i(\beta)=0$. In this case $\beta$ is an arbitrary constant in $\C^*$
and $\gamma_{ab}$ are arbitrary.

{\bf 2nd case, some $\paa_i \paa_j \paa_k\Psi\neq 0$:}
then \eqref{2.30} says $\beta =1$. Now \eqref{2.31} becomes
$(\paa_i\circ\paa_j)(\varphi_{2n+2})=0$.
\end{proof}

\section{(TEP)-structures}\label{c3}
\setcounter{equation}{0}

\subsection{Definitions}\label{c3.1}

\noindent
For the proof of theorem \ref{c2.2} (b) we need a datum which is 
between the Frobenius manifold and its F-manifold, namely
the (TEP)-structure on $(\pi^*TM)_{|\C\times M}$ where
$\pi:\pmmm\to M$ is the projection.
We will show that this structure does not depend on 
$(U_\bullet,\lambda)$. 
Then theorem \ref{c2.2} (b) will follow easily.

A {\it (TEP)-structure} of weight $w\in \Z$ consists of data
$(H\to\C\times M,\nnn,S)$. Here $M$ is a complex manifold,
$H\to\C\times M$ is a holomorphic vector bundle,
$\nnn$ is a flat connection on $H_{|\C^*\times M}$ with a pole
of Poincar\'e rank 1 along $\{0\}\times M$,
and $P$ is a $\nnn$-flat $(-1)^w$-symmetric nondegenerate pairing
\begin{eqnarray*}
P: H_{(z,t)}\times H_{(-z,t)}\to \C
\quad \textup{ for }\quad (z,t)\in \C^*\times M
\end{eqnarray*}
which extends with $j:(z,t)\mapsto (-z,t)$ to a nondegenerate pairing
\begin{eqnarray}\label{3.1}
P:\OO(H)\otimes j^*\OO(H)\to z^w\OO_{\C\times M}.
\end{eqnarray}

A {\it (TLEP)-structure} of weight $w\in\Z$ is an extension
of the bundle $H\to\C\times M$ of a (TEP)-structure 
to a holomorphic vector bundle
$\whh H\to\pmmm$ such that the pole along $\immm$ is logarithmic
and such that $P$ extends to an everywhere nondegenerate pairing 
from $\OO(\whh H)\otimes j^*\OO(\whh H)$ to 
$z^w\OO_{\pmmm}$.

A (trTLEP)-structure is a (TLEP)-structure 
such that $\whh H$ is a family of trivial bundles on $\P^1$.

A (TEP)-structure induces a Higgs field
$$C=[z\nnn]:\OO(H_{|\nmmm})\to \OO(H_{|\nmmm})\otimes\oomm^1_M,$$
an endomorphism
$$U=[z^2\nnn_{\paa_z}]:\OO(H_{|\nmmm})\to \OO(H_{|\nmmm})$$
with $[C,U]=0$ 
and a symmetric nondegenerate pairing
$$g=[z^{-w}P]:\OO(H_{|\nmmm})\times\OO(H_{|\nmmm})\to \OO_M$$
with $C^*=C$ and $U^*=U$.

A (trTLEP)-structure is equivalent to a differential 
geometric structure on $H_{|\nmmm}$ containing $C$ and $g$ and 
more data, which is called {\it Frobenius type structure} in 
\cite[ch. 4]{HERT-MANI:2004}, the equivalence is stated in \cite[VI 7]{SABB:2002}, \cite[\S 5.2]{HERT:2003} and \cite[\S 4.2]{HERT-MANI:2004}.

There is a 1-1-correspondence between extensions of (TEP)-structures
to (TLEP)-structures and monodromy invariant filtrations of the 
space
$H^\infty:=\{\textup{global flat multivalued sections in }
H_{|\csmmm}\}$ ($M$ is small and contractible), cf.\ 
\cite[III.1.4]{SABB:2002} and \cite[\S 8.2]{HERT:2002}.
So, obtaining extensions of (TEP)-structures to (TLEP)-structures is easy.
There exist examples (already known by Birkhoff) of (TEP)-structures
such that none of these extensions are (trTLEP)-structures.
But the (TEP)-structures of interest for us have nice extensions
to (trTLEP)-structures.

\subsection{Two examples}\label{c3.2} For the constructions in this paper,  the following
two examples of (TEP)-structures play an important role:

\medskip
 
(i) If $M$ is a Frobenius manifold  with 
$\Lie_E(g)=(2-d)\cdot g$, $d\in\C$, and $\pi:\pmmm\to M$ is the projection,
then $\pi^*T_M$ is canonically equipped with a 
(trTLEP)-structure of (any) weight $w\in\Z$ 
(Dubrovin, Manin, e.g. \cite[ch. 4]{HERT-MANI:2004}):
\begin{eqnarray}\label{3.2}P &:=& z^w\cdot(\id,j)^*g,\\ \label{3.3}
\nnn &:=& \pi^*\nnn^g + \frac{1}{z}C + 
\left(-\frac{1}{z}E\circ -\nnn^g E+\frac{2-d+w}{2}\id\right)
\frac{\ddd z}{z},\\ \nonumber
&& \textup{where }C\textup{ is the Higgs field on }TM\textup{ with }
C_X=X\circ.
\end{eqnarray}
(Compared to \cite[ch. 4]{HERT-MANI:2004}, here we changed the sign in $C_X=X\circ$
and used $-z$ instead of $z$ in \eqref{3.3}, in order to be compatible
with section \ref{c2}. The signs there are chosen to make the 
comparison with projective special geometry in section \ref{c6} smoother.)

\medskip

(ii) Let $(M,\mathcal{V},\nnn^\mathcal{V},S,F^\bullet,w)$ be a variation of Hodge like
filtrations with pairing ($M$ is small and contractible).
Let $\pi_\C:\cmmm\to M$ be the projection and 
$\pi_\C^*\nnn^\mathcal{V}$ be the flat connection on $\pi_\C^*\mathcal{V}$ whose flat
sections are the pull backs of $\nnn^\mathcal{V}$-flat sections in $\mathcal{V}$.
Define a bundle $H\to\cmmm$ with 
$H_{|\csmmm} = \pi_\C^*\mathcal{V}_{|\csmmm}$ by
\begin{eqnarray}\label{3.4}
\OO(H):=\sum_{p\in\Z}z^{w-p}\cdot \OO(\pi_\C^*F^p),
\end{eqnarray}
a flat connection $\nnn$ on $H_{|\csmmm}$ by
\begin{eqnarray}\label{3.5}
\nnn:=\pi_\C^*\nnn^\mathcal{V}
\end{eqnarray}
and a pairing 
$P:(\pi_\C^*\mathcal{V})_{(z,t)}\times (\pi_\C^*\mathcal{V})_{(-z,t)}\to \C$ for $(z,t)\in\csmmm$ by 
\begin{eqnarray}\label{3.6}
P(\pi_\C^*a,\pi_\C^*b):= \frac{1}{(2\pi i)^w}\cdot S(a,b).
\end{eqnarray}
{\bf Claim:} {\it Then $(H\to\cmmm,\nnn,P)$ is a (TEP)-structure
of weight $w$.}

On the one hand, this follows by unwinding the construction behind
corollary 7.14 (b) in \cite{HERT:2003}
(an extra factor $\frac{1}{(2\pi i)^w}$ in \eqref{3.6} makes the 
definitions here compatible with \cite{HERT:2003}).

On the other hand, it can be seen directly as follows.
\eqref{3.4}, \eqref{3.5} and the Griffiths transversality
\eqref{1.1} show
$z\nnn_{\paa_z}\OO(H)\subset\OO(H)$ and $z\nnn_X\OO(H)\subset\OO(H)$ for $X\in T_M$.
This gives the pole of Poincar\'e rank 1 along $\nmmm$
(even $z^2\nnn_{\paa_z}\OO(H)\subset \OO(H)$ would be sufficient).

The conditions \eqref{1.4}, \eqref{3.4}, \eqref{3.6} and the 
nondegenerateness of $S$ show that $P$ maps
$\OO(H)\otimes\OO(H)$ to $z^w\OO_\cmmm$ and that this map
is nondegenerate.
Obviously, $P$ is $(-1)^w$-symmetric and $\nnn$-flat.

\subsection{F-manifolds from (TEP)-structures}\label{c3.3}
There is a construction of Frobenius manifolds from 
meromorphic connections which goes back to the construction
of Frobenius manifolds in singularity theory by M. Saito \cite{SAIT:1989} and 
K. Saito. It is formalized in 
\cite[Th\'eor\`eme VII.3.6]{SABB:2002}\cite{BARA:2002}\cite{BARA:2001} and 
\cite[theorems 4.2 and 4.5]{HERT-MANI:2004}.
In \cite{HERT-MANI:2004} the initial data are a (trTLEP)-structure
with a distinguished section and an isomorphy condition.
The following result gives the construction of a weaker
datum, an F-manifold, from a weaker initial datum, 
a (TEP)-structure with an isomorphy condition.
The proof relies on \cite[4.1]{HERT:2003}.

\begin{theorem}
Let $(H\to \C\times M,\nnn,P)$ be a (TEP)-structure
(actually, the pairing $P$ will not be used) with 
Higgs field $C=[z\nnn]$ and endomorphism 
$U=[z^2\nnn_{\paa_z}]$ on $\OO(H_{|\nmmm})$.
Then $\OO(H_{|\nmmm})$ is a $T_M$-module.
Suppose that the following isomorphy condition holds:
\begin{eqnarray}\label{3.10}
\OO(H_{|\nmmm}) \textup{ is a free }T_M\textup{-module
of rank }1.
\end{eqnarray}
Then there is a unique multiplication $\circ$ on $T_M$
with $C_{X\circ Y}=C_XC_Y$ and a unique unit field $e$.
The multiplication is commutative and associative.
The unit field satisfies $C_e=\id$. There is also a unique vector field
$E$ with $C_E=-U$. The tuple $(M,\circ,e,E)$ is an F-manifold
with Euler field
\end{theorem}

\begin{proof}
The first part of the proof follows \cite[lemma 4.1]{HERT:2003}.
Locally a section $\xi$ in $H_{|\nmmm}$ is chosen such that
$C_\bullet \xi :TM\to H_{|\nmmm}$ is an isomorphism.
The multiplication $\circ$ and the vector fields $e$ and 
$E$ are defined by
\begin{eqnarray*}
C_{X\circ Y}\xi = C_XC_Y\xi,\quad 
C_e\xi = \xi,\quad C_E\xi = -U\xi.
\end{eqnarray*}
The multiplication is commutative and associative, and
$e$ is a unit field.

Because of 
\begin{eqnarray*}
C_{X\circ Y}C_Z\xi = C_XC_YC_Z\xi,\quad 
C_eC_Z\xi = C_Z\xi,\quad C_EC_Z\xi = -UC_Z\xi
\end{eqnarray*}
the multiplication and the vector fields $e$ and $E$
are independent of the choice of $\xi$ and satisfy
\begin{eqnarray*}
C_{X\circ Y} = C_XC_Y,\quad 
C_e=\id,\quad C_E = -U.
\end{eqnarray*}
The proof that they give an F-manifold with Euler field
will use \cite[lemma 4.3]{HERT:2003}. In order to apply it, it
would be nice to extend the (TEP)-structure to a 
(trTLEP)-structure. That is not always possible,
but by \cite[lemma 2.7]{HERT-MANI:2004} one can change and extend
it (locally in $M$) to the following weaker structure:
A holomorphic vector bundle $\whh H\to\pmmm$ such that
$\whh H_{|(\C-\{1\})\times M} = H_{|(\C-\{1\})\times M}$, 
such that the connection $\nnn$ has logarithmic poles
along $\{1\}\times M$ and $\nmmm$ and such that $\whh H$
is a family of trivial bundles on $\P^1$
(here $M$ is supposed to be small).
Because of the last condition
\begin{eqnarray}\label{3.11}
\OO(\whh H_{|\nmmm})\cong \OO(\whh H_{|\immm})\cong
\pi_*\OO(\whh H)
\end{eqnarray}
and $C$ and $U$ on $\whh H_{|\nmmm}$ as well as the residual
connection $\nnn^{res}$ on $\whh H_{|\immm}$ are shifted 
to the isomorphic sheaves.
There are two further endomorphisms $V$ and $W$ 
(with $V+W=-$ residue endomorphism on $H_{|\immm}$)
such that for fiberwise global sections 
$\sigma\in \pi_*\OO(\whh H)$
\begin{eqnarray}\label{3.12}
\nnn\sigma = \left(\nnn^{res}+\frac{1}{z}C +
(\frac{1}{z}U+V+\frac{z}{z-1}W)\frac{\ddd z}{z}\right)
\sigma.
\end{eqnarray}
The flatness of $\nnn$ yields $\nnn^{res}(C)=0$ and 
$\nnn^{res}(U)-[C,V]+C=0$. 
Therefore lemma 4.3 in \cite{HERT:2003} applies and shows that
$(M,\circ,e,E)$ is an F-manifold with Euler field.
\end{proof}

\subsection{The classifying space $\check D_{PHS}$}\label{c3.4}
For the rest of this section, a variation of Hodge like
filtrations $((B_0,0),\mathcal{V},\nnn,S,F^\bullet)$ of weight 
$w=3$ with pairing and CY-condition 
\eqref{1.7} \& \eqref{1.8} is fixed.
$B_0$ is a (sufficiently small) representative of a 
germ $(B_0,0)$ of a manifold of dimension $n$. For $b\in B_0$
the filtration is denoted $F^\bullet_b$. By abuse of notation
we also denote its $\nnn$-flat shift to the fiber
$\mathcal{V}_0$ by $F^\bullet_b$.

There is a classifying space $\check D_{PHS}$ for all
Hodge like filtrations with the same discrete data as
$F^\bullet_0$,
\begin{equation}\label{3.13}
\begin{split}
\check D_{PHS} &:= \big\{\textup{filtrations }F^\bullet
\textup{ on }\mathcal{V}_0\ |\ S(F^p,F^{4-p})=0,\\
&\hspace{0.8Cm}0=F^4\subset F^3\subset F^2\subset F^1\subset F^0=\mathcal{V}_0,\\
& \hspace{1.3cm}\dim F^3=1=\dim F^0/F^1, \\
&\hspace{1.8cm}\dim F^2/F^3=n=\dim F^1/F^2\big\}.
\end{split}
\end{equation}
It goes back to work of Griffiths and Schmid
(see also \cite{BRYA-GRIF:1983}). It is a complex homogeneous space
and a projective manifold. More concretely, it is a 
bundle over the lagrangian Grassmannian
\begin{eqnarray}\label{3.14}
\check D_{lag} =\{F^2\subset \mathcal{V}_0\ |\ \dim F^2=n+1,\ 
S(F^2,F^2)=0\}
\end{eqnarray}
with fibers $\P(F^2)\cong\P^n$. The base $\check D_{lag}$
has dimension $n(n+1)/2$, the fibers contain the
possible choices of $F^3\subset F^2$ and 
$F^1=(F^3)^{\perp_S}$. The natural period map
\begin{eqnarray}\label{3.15}
\Pi:B_0\to \check D_{PHS},\quad b\mapsto F^\bullet_b,
\end{eqnarray}
is horizontal. Because of the CY-condition it is 
an embedding. It determines the variation of 
Hodge like filtrations.

\subsection{The classifying space $\check D_{BL}$}\label{c3.5}
There is a classifying space $\check D_{BL}$ for certain
(TEP)-structures with a natural projection 
$\pi_{BL}:\check D_{BL}\to \check D_{PHS}$. 
In a more general setting such spaces have been constructed in 
\cite{HERT:1999} and taken up again in \cite{HERT-SEVE:2008}\cite{HERT-SEVE2:2008}.
Here we restrict to the special case which we need.
Before defining and discussing $\check D_{BL}$, some notations
have to be established.

$\mathcal{V}_0$ is a $2n+2$ dimensional complex vector space with 
antisymmetric and nondegenerate pairing $S$.
The vector bundle $H':=\mathcal{V}_0\times \C^*$ comes equipped with 
the trivial flat connection $\nnn$ and a pairing
\begin{eqnarray}\label{3.16}
P:H_z'\times H_{-z}'&\to&\C\quad \textup{ for }z\in \C^*\\
(a,b)&\mapsto&S(a,b),
\quad \textup{ here }a,b\in H'_z=\mathcal{V}_0=H'_{-z}.\nonumber
\end{eqnarray}
It is $\nnn$-flat, antisymmetric and nondegenerate.
The space of global flat sections in $H'$ is denoted
$C^0$. It is identified with $\mathcal{V}_0$. 
For $\alpha\in\Z$ and $a\in \mathcal{V}_0$ the section 
$(z\mapsto z^\alpha\cdot a(z))$ is denoted $z^\alpha a$,
the space of such sections is denoted 
$C^\alpha=z^\alpha\cdot C^0$.

The space $V^\alpha:=\C\{z\}\cdot C^\alpha$ is the germ at 0
of the Deligne extension of $H'\to\C^*$ to a vector bundle
on $\C$ with logarithmic pole at 0 with $\alpha$ as the 
only eigenvalue of the residue endomorphism
$[\nnn_{z\paa_z}]$. 
Together the spaces $V^\alpha$, $\alpha\in\Z$, form the 
Kashiwara-Malgrange $V$-filtration. 
Of course $\Gr_V^\alpha \cong C^\alpha$ canonically.

Any (TEP)-structure $(H\to\C,\nnn,P)$ with 
$H_{|\C^*}=H'$ is determined by the germ $\HH_0:=\OO(H)_0$
at 0. We are interested in the regular singular (TEP)-structures,
i.e. those with $\HH_0\subset \sum_\alpha V^\alpha$.
The spectrum of such a (TEP)-structure is the tuple
$(\alpha_1,...,\alpha_{2n+2})\in \Z^{2n+2}$ with 
$\alpha\leq ...\leq \alpha_{2n+2}$ and
\begin{eqnarray}\label{3.17}
\sharp(i\ |\ \alpha_i=\alpha) 
= \dim \Gr_V^\alpha \HH_0/\Gr_V^\alpha z\HH_0.
\end{eqnarray}
The (TEP)-structure induces a decreasing filtration $F^\bullet(H)$
on $\mathcal{V}_0$ by
\begin{eqnarray}\label{3.18}
F^p(H)=F^p(\HH_0):= z^{p-3}\Gr_V^{3-p}\HH_0\subset C^0=\mathcal{V}_0.
\end{eqnarray}
The classifying space $\check D_{BL}$ of (TEP)-structures
relevant for us is
\begin{equation}\label{3.19}
\begin{split}
\check D_{BL}=& \big\{\textup{regular singular (TEP)-structures}
(H,\nnn,P)\\
&~\textup{of weight $3$ with }H_{|\C^*}=H' \textup{ and spectrum } \\
&~(\alpha_1,...,\alpha_{2n+2})=(0,1,...,1,2,...,2,3)\big\}\\
\end{split}
\end{equation}
with $1$ and $2$ each $n$ times.
\begin{theorem}
$\check D_{BL}$ is an algebraic manifold and a bundle
on $\check D_{PHS}$ via
\begin{eqnarray}\label{3.20}
\pi_{BL}:\check D_{BL}\to \check D_{PHS},\quad
H\mapsto F^\bullet(H).
\end{eqnarray}
The fibers are isomorphic to $\C^{n+1}$ as affine algebraic
manifolds and carry a good $\C^*$-action with weights 
$(1,...,1,2)$. The corresponding zero section
$\check D_{PHS}\hookrightarrow \check D_{BL}$ is given
by the (TEP)-structures defined as in \eqref{3.4}.
\end{theorem}
\begin{proof}
This theorem is a special case of \cite[theorem 5.6]{HERT:1999},
but here the proof simplifies.
In the following we present the proof, as it provides useful
explicit control on $\check D_{BL}$.
\begin{lemma} 
$F^\bullet(H)\in \check D_{PHS}$ if $H\in\check D_{BL}$.
\end{lemma}
\begin{proof}
$F^\bullet(H)$ is decreasing because
\begin{eqnarray*}
F^{p+1}(H)&=&z^{p+1-3}\Gr_V^{3-(p+1)}\HH_0\\
&=&z^{p-3}\Gr_V^{3-p}z\HH_0
\subset z^{p-3}\Gr_V^{3-p}\HH_0=F^p(H).
\end{eqnarray*}
Because of the spectral numbers
\begin{eqnarray*}
(\dim F^p(H)\ |\ p=3,2,1,0)
=(1,n+1,2n+1,2n+2).
\end{eqnarray*}
If $a_1\in F^p(H)$ and $a_2\in F^{4-p}(H)$ then there are 
sections
\begin{eqnarray*}
\sigma_1\in \HH_0\cap(z^{3-p}a_1+V^{4-p})
\textup{ and }
\sigma_2\in \HH_0\cap(z^{3-(4-p)}a_2+V^{4-(4-p)}).
\end{eqnarray*}
The $z^2$-coefficient of $P(\sigma_1,\sigma_2)\in 
z^3\C\{z\}$ vanishes. This shows $S(a_1,a_2)=0$.
Therefore $S(F^p(H),F^{4-p}(H))=0$ and 
$F^\bullet(H)\in \check D_{PHS}$
\end{proof}

Now the fiber $\pi_{BL}^{-1}(F^\bullet)$ for an arbitrary
$F^\bullet$ shall be determined. The (TEP)-structures
in this fiber will be described by certain distinguished
sections in them. For that we make the same choices
as in lemma \ref{c2.1}, a filtration $U_\bullet$ which 
is opposite to $F^\bullet$ and 
a basis $v_1,...,v_{2n+2}$ of $\mathcal{V}_0$ which satisfies 
\eqref{2.1} and \eqref{2.2}.
We use again the convention \eqref{2.11} for indices
$i,j,k,a,b,\alpha$. We define sections
\begin{eqnarray}\label{3.21} 
\left.\begin{matrix}s_1&=&v_1\in C^0,& s_i&=&zv_i\in C^1,\\
s_a&=&z^2v_a\in C^2,& s_{2n+2}&=&z^3v_{2n+2}\in C^3,\end{matrix}\right\}
\end{eqnarray}
and we define $(p_1,p_i,p_a,p_{2n+2})=(3,2,1,0)$ so that
$v_\alpha\in F^{p_\alpha}\cap U_{p_\alpha}$ and 
$s_\alpha\in z^{3-p_\alpha}\cdot F^{p_\alpha}\cap U_{p_\alpha}$.
The following picture illustrates this and the
next lemma.

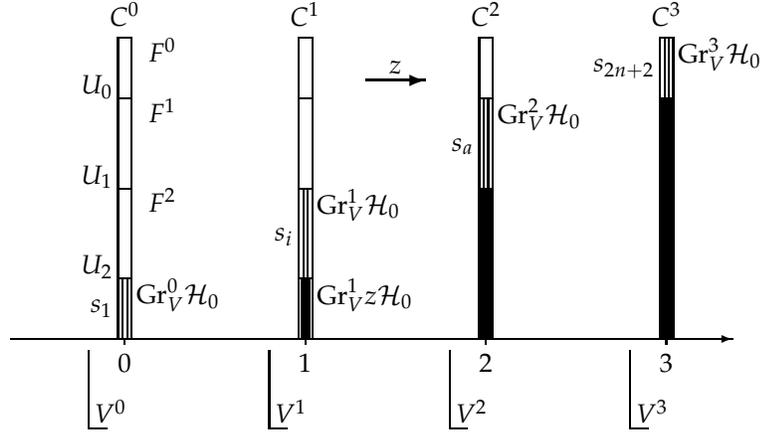
\setlength{\unitlength}{0.8mm}

\noindent
\begin{picture}(140,80)
\thinlines

\put(10,20){\vector(1,0){120}}


\put(28,20){\framebox(2.1,50){}}
\put(58,20){\framebox(2.1,50){}}
\put(88,20){\framebox(2.1,50){}}
\put(118,20){\framebox(2.1,50){}}

\put(28,30){\line(1,0){2}}
\put(58,30){\line(1,0){2}}
\put(88,30){\line(1,0){2}}
\put(118,30){\line(1,0){2}}

\put(28,45){\line(1,0){2}}
\put(58,45){\line(1,0){2}}
\put(88,45){\line(1,0){2}}
\put(118,45){\line(1,0){2}}


\put(28,60){\line(1,0){2}}
\put(58,60){\line(1,0){2}}
\put(88,60){\line(1,0){2}}
\put(118,60){\line(1,0){2}}

\put(31,30){\makebox(0,0)[tl]{${\rm Gr}_V^{0}
{\mathcal H}_0$}}
\put(61,45){\makebox(0,0)[tl]{${\rm Gr}_V^{1}
{\mathcal H}_0$}}
\put(61,30){\makebox(0,0)[tl]{${\rm Gr}_V^{1}
z{\mathcal H}_0$}}
\put(91,60){\makebox(0,0)[tl]{${\rm Gr}_V^{2}
{\mathcal H}_0$}}
\put(121,70){\makebox(0,0)[tl]{${\rm Gr}_V^{3}
{\mathcal H}_0$}}

\multiput(28.7,20)(0.7,0){2}{\line(0,1){10}}
\multiput(58.7,20)(0.7,0){2}{\line(0,1){25}}
\multiput(58.35,20)(0.7,0){3}{\line(0,1){10}}
\multiput(88.7,20)(0.7,0){2}{\line(0,1){40}}
\multiput(88.35,20)(0.7,0){3}{\line(0,1){25}}
\multiput(118.7,20)(0.7,0){2}{\line(0,1){50}}
\multiput(118.35,20)(0.7,0){3}{\line(0,1){40}}

\put(29,16){\makebox(0,0){$0$}}
\put(59,16){\makebox(0,0){$1$}}
\put(89,16){\makebox(0,0){$2$}}
\put(119,16){\makebox(0,0){$3$}}

\put(29,74){\makebox(0,0){$C^{0}$}}
\put(59,74){\makebox(0,0){$C^{1}$}}
\put(89,74){\makebox(0,0){$C^{2}$}}
\put(119,74){\makebox(0,0){$C^{3}$}}

\put(29,19){\line(0,1){1}}
\put(59,19){\line(0,1){1}}
\put(89,19){\line(0,1){1}}
\put(119,19){\line(0,1){1}}

\put(23,5){\line(1,0){3}}
\put(23,5){\line(0,1){13}}
\put(24,6){\makebox(0,0)[bl]{$V^{0}$}}

\put(53,5){\line(1,0){3}}
\put(53,5){\line(0,1){13}}
\put(54,6){\makebox(0,0)[bl]{$V^{1}$}}

\put(83,5){\line(1,0){3}}
\put(83,5){\line(0,1){13}}
\put(84,6){\makebox(0,0)[bl]{$V^{2}$}}

\put(113,5){\line(1,0){3}}
\put(113,5){\line(0,1){13}}
\put(114,6){\makebox(0,0)[bl]{$V^{3}$}}

\put(27,25){\makebox(0,0)[mr]{$s_1$}}
\put(57,37){\makebox(0,0)[mr]{$s_i$}}
\put(87,52){\makebox(0,0)[mr]{$s_a$}}
\put(117,65){\makebox(0,0)[mr]{$s_{2n+2}$}}

\put(27,30){\makebox(0,0)[br]{$U_2$}}
\put(27,45){\makebox(0,0)[br]{$U_1$}}
\put(27,60){\makebox(0,0)[br]{$U_0$}}
\put(33,45){\makebox(0,0)[tl]{$F^2$}}
\put(33,60){\makebox(0,0)[tl]{$F^1$}}
\put(33,70){\makebox(0,0)[tl]{$F^0$}}

\thicklines
\put(69,63){\vector(1,0){10}}
\put(74,64){\makebox(0,0)[b]{$z$}}

\end{picture}

\begin{lemma}
(a) For any $H\in \pi_{BL}^{-1}(F^\bullet)$ there exist 
unique sections $\sigma_\alpha\in\HH_0$ with 
$\sigma_\alpha-s_\alpha\in\sum_{\beta>3-p_\alpha}z^\beta\cdot
U_{2-\beta}$.
They form a $\C\{z\}$-basis of $\HH_0$.
Explicitly, they take the form
\begin{eqnarray}\label{3.22}
\sigma_1&=& s_1+\sum_a y_a\cdot z^{-1}\cdot s_a
+y_{2n+2}\cdot z^{-1}\cdot s_{2n+2},\\
\sigma_i&=& s_i+y_{n+i}\cdot z^{-1}\cdot s_{2n+2},
\nonumber\\
\sigma_a&=& s_a,\nonumber\\
\sigma_{2n+2}&=& s_{2n+2},\nonumber
\end{eqnarray}
with some $y_a\in\C,\ y_{2n+2}\in\C.$

(b) The other way round, for any $y_a\in\C$ and 
$y_{2n+2}\in C$, these sections generate over $\C\{z\}$
the germ $\HH_0$ of a (TEP)-structure in 
$\pi_{BL}^{-1}(F^\bullet)$.

(c) Therefore $\pi_{BL}^{-1}(F^\bullet)\cong \C^{n+1}$
as an affine algebraic manifold, and 
$y_{n+2},...,y_{2n+1},y_{2n+2}$ are coordinates on it.
\end{lemma}

\begin{proof}
(a) Because of the spectral numbers 
$\HH_0=(\HH_0\cap (C^0+C^1+C^2))\oplus V^3$.
Because of $F^\bullet (H)=F^\bullet$ there exist
sections in $\HH_0\cap (s_\alpha+V^{4-p_\alpha})$.
Existence and uniqueness of the sections $\sigma_\alpha$
is now an easy argument in linear algebra. 
It is also clear that they form a $\C\{z\}$-basis 
of $\HH_0$. A priori they take the form
\begin{eqnarray}\label{3.23}
\sigma_1&=& s_1+\sum_a y_a\cdot z^{-1}\cdot s_a
+ x_{2n+2}\cdot z^{-2}\cdot s_{2n+2}\\
&&\hspace*{3cm} +\ y_{2n+2}\cdot z^{-1}\cdot s_{2n+2},\nonumber \\
\sigma_i&=& s_i+x_{n+i}\cdot z^{-1}\cdot s_{2n+2},\nonumber\\
\sigma_a&=& s_a,\nonumber\\
\sigma_{2n+2}&=& s_{2n+2},\nonumber
\end{eqnarray}
with $y_a, x_{2n+2}, y_{2n+2}, x_a\in\C.$
The germ $\HH_0$ satisfies 
\begin{eqnarray}\label{3.24}
&& z\nnn_{z\paa_z} \HH_0\subset \HH_,\\
&& P:\HH_0\times \HH_0\to z^3\C\{z\}
\quad\textup{ nondegenerate}.\label{3.25}
\end{eqnarray}
On the other hand, the sections $\sigma_\alpha$
satisfy
\begin{eqnarray}\label{3.26}
z\nnn_\zdz \sigma_1 &=&
\sum_ay_a\cdot \sigma_a + x_{2n+2}\cdot z^{-1}\cdot
\sigma_{2n+2} + 2y_{2n+2}\cdot \sigma_{2n+2},\\
z\nnn_\zdz \sigma_i &=& z\cdot \sigma_i 
+ x_{2n+2}\cdot \sigma_{2n+2}, \label{3.27},\\
z\nnn_\zdz \sigma_a &=& 2z\cdot \sigma_a,\label{3.27b}\\ 
z\nnn_\zdz \sigma_{2n+2} &=& 3z\cdot \sigma_{2n+2},
\label{3.28},
\end{eqnarray}
\begin{equation}
P(\sigma_\alpha,\sigma_\beta)
=
\begin{pmatrix} 2zx_{2n+2} & z^2(y_{n+i}-x_{n+i}) & 
 0 & z^3 \\
 z^2(x_{n+i}-y_{n+i}) & 0 & z^3 & 0 \\
 0 & z^3 & 0 & 0 \\
 z^3 & 0 & 0 & 0 \end{pmatrix}.\label{3.29}
 \end{equation}
 with $\alpha\in\{1,i,a,2n+2\}$ and $\beta\in\{1,j,b,2n+2\}$.
Both \eqref{3.26} and \eqref{3.29} show $x_{2n+2}=0,$
\eqref{3.29} shows also $x_a=y_a$.
This proves part (a).

(b) The sections $\sigma_\alpha$ generate over 
$\C\{z\}$ the germ $\HH_0$ of the sections of a 
vector bundle $H\to\C$ which extends $H'\to\C^*$.
Because of \eqref{3.26} - \eqref{3.29} $\HH_0$
satisfies \eqref{3.24}\&\eqref{3.25}. 
Therefore $(H,\nnn,P)$ is a (TEP)-structure.
It is in $\pi_{BL}^{-1}(F^\bullet)$.

(c) It is now also clear.
\end{proof}

\bigskip
There is a natural $\C^*$-action on $\check D_{BL}$
which respects the fibers of $\pi_{BL}$.
It is defined (coordinate independently) as follows.
For any $r\in\C^*$ define 
$\pi_r:\C\to\C, z\mapsto r\cdot z$.
Then $(H,\nnn,P)\in \pi_{BL}^{-1}(F^\bullet)$
is mapped by $r\in \C^*$ via the $\C^*$-action
to $\pi_r^*(H,\nnn,P)\in\pi_{BL}^{-1}(F^\bullet)$.

This action works as follows on the sections and 
coordinates in the last lemma.
If $a\in C^0$ and $\alpha\in\Z$ then 
$\pi_r^*(z^\alpha\cdot a)=r^\alpha\cdot z^\alpha\cdot a$,
so
\begin{eqnarray*}
\pi_r^*\sigma_1&=& s_1+\sum_a r_\cdot y_a\cdot z^{-1}\cdot s_a
+r^2\cdot y_{2n+2}\cdot z^{-1}\cdot s_{2n+2},\\
\pi_r^*\sigma_i&=& r\cdot 
\left(s_i+r\cdot y_{n+i}\cdot z^{-1}\cdot s_{2n+2}\right),\\
\pi_r^*\sigma_a&=& r^2\cdot \sigma_a,\\
\pi_r^*\sigma_{2n+2}&=& r^3\cdot \sigma_{2n+2},
\end{eqnarray*}
and the $\C^*$-action on $\pi_{BL}^{-1}(F^\bullet)$
is given in the coordinates $(y_a,y_{2n+2})$ by 
$r.(y_a,y_{2n+2})=(r\cdot y_a,r^2\cdot y_{2n+2})$.
This finishes the proof of the theorem.
\end{proof}

\begin{remark}
(i) Sections like the $\sigma_\alpha$ above were used first
in \cite[ch. 3]{SAIT:1989}.

(ii) The vector bundle $\mathcal{V}_0\times \check D_{PHS}$ carries
the trivial flat connection $\nnn^{\mathcal{V}_0\times D}$ and 
the tautological filtration $F^\bullet$. The filtration is 
a family of Hodge like filtrations, but not a variation,
because the Griffiths transversality is violated.
Nevertheless, a filtration $U_\bullet$ which is opposite
to a reference filtration $F^\bullet_*\in \check D_{PHS}$
and then also to all filtrations nearby,
induces a decomposition 
$\nnn^{\mathcal{V}_0\times D}=\nnn^U+C^U$ into a flat connection
$\nnn^U:\OO(F^p\cap U_p)\to \OO(F^p\cap U_p)\otimes
\Omega^1_{(\textup{nbhd of }F^\bullet_*)}$ and a tensor
$C^U:\OO(F^p\cap U_p)\to \OO(U_{p-1})\otimes
\Omega^1_{(\textup{nbhd of }F^\bullet_*)}$.

A basis $v_1(*),...,v_{2n+2}(*)$ of vectors with 
\eqref{2.1} and \eqref{2.2} for $F^\bullet_*$
extends to a $\nnn^{\mathcal{V}_0\times D}$-flat basis of sections
of $\mathcal{V}_0\times \check D_{PHS}$ with \eqref{2.1} and
\eqref{2.2}.
The formulas \eqref{3.21} and \eqref{3.22} extend to 
these sections and yield a trivialization of the bundle
$\pi_{BL}:\check D_{BL}\to \check D_{PHS}$ on 
$(\textup{nbhd of }F^\bullet_*)\subset \check D_{PHS}$,
with fiber coordinates $y_a,y_{2n+2}$.

(iii) The vector field on $\check D_{BL}$ which generates
the canonical $\C^*$-action is denoted $E_{BL}$.
It is tangent to the fibers of $\pi_{BL}$. 
In local coordinates as in (ii) it is 
$E_{BL}=\sum_a y_a\frac{\paa}{\paa y_a}+
2y_{2n+2}\frac{\paa}{\paa y_{2n+2}}$.
The zero section 
$\check D_{PHS}\hookrightarrow \check D_{BL}$ consists
of the (TEP)-structures with 
$\HH_0=\sum_p \C\{z\}\cdot z^{3-p}\cdot F^p$,
$F^\bullet\in \check D_{PHS}$ (as in \eqref{3.4}).

(iv) Any (TEP)-structure in $\check D_{BL}$ is
determined by $L:=\HH_0\cap (C^0+C^1+C^2)$.
The $z^2$-coefficient of $P$ restricts to a symplectic
form $P^{(2)}$ on $C^0+C^1+C^2$.
The multiplication by $z$ restricts to a nilpotent 
endomorphism $\mu_z:C^0+C^1+C^2\to C^1+C^2$ with 
$C^0\stackrel{\mu_z}{\to}C^1\stackrel{\mu_z}{\to}
C^2\stackrel{\mu_z}{\to}0.$
We leave it to the reader to show that the classifying
space $\check D_{BL}$ can be identified with the following
classifying space of certain lagrangian subspaces,
\begin{equation}\label{3.30}
\begin{split}
\www D_{BL} =& ~\big\{L\subset C^0+C^1+C^2\ |\ 
\mu_z(L)\subset L, \mu_z(\nnn_\zdz L)\subset L,\\
&\hspace{0.3cm}P^{(2)}(L,L)=0,\dim L=3n+3,\nonumber\\
&\hspace{0.3cm}\dim L\cap (C^1+C^2)=3n+2,\dim L\cap C^2=2n+1\big\}.\hspace*{1cm}
\end{split}
\end{equation}
\end{remark}

\subsection{The canonical (TEP)-structure 
with isomorphy condition}\label{c3.6}

As in section \ref{c3.4} a variation of Hodge like filtrations
$((B_0,0),\mathcal{V},\nnn^\mathcal{V},S,F^\bullet)$ of weight $w=3$ with pairing 
and CY-condition \eqref{1.7} \& \eqref{1.8} is fixed.
The base space $B_0$ is identified with its image 
$\Pi(B_0)\subset \check D_{PHS}$ under the period map
$\Pi:B_0\to \check D_{PHS}$ in \eqref{3.15}.
Define
\begin{eqnarray}\label{3.31}
B_2:= \pi_{BL}^{-1}(\Pi(B_0)) \quad \textup{and}\quad 
M:=\C\times B_2.
\end{eqnarray}
The coordinate on the factor $\C$ in $\C\times B_2$ 
is denoted $y_1$. 
The tautological family of (TEP)-structures on
$\check D_{BL}$ restricts to a family of (TEP)-structures
on $B_2$. We extend it to a family 
$(H\to\C\times B_2,\nnn,P)$ of (TEP)-structures on $M$
by twisting all sections with $e^{y_1/z}$.

\begin{theorem}With these definitions:
\begin{itemize}
\item[(a)] This is a (TEP)-structure on $M$ with isomorphy 
condition \eqref{3.10}. Theorem \ref{c3.1} applies and 
gives $M$ a canonical F-manifold structure.
The unit field is $e=\frac{\paa}{\paa y_1}$, the Euler
field is $E=y_1\frac{\paa}{\paa y_1}-(E_{BL})_{|B_2}$
($E_{BL}$ is defined in remark \ref{c3.5} (iii)).
\item[(b)] For any of the Frobenius manifolds in Theorem \ref{c2.2},
the underlying manifold $M^{U,\lambda}$ is 
canonically isomorphic to $M$. The isomorphism
respects the F-manifold structure and the Euler field.
\end{itemize}
\end{theorem}
\begin{proof}
It will be proved in several steps. For the rest of the
section an opposite filtration $U_\bullet$ and a 
vector $\lambda\in F^3_0-\{0\}$ as in theorem \ref{c2.2}
are chosen. Furthermore, sections $v_1,...,v_{2n+2}$ on $\mathcal{V}$
as in lemma \ref{c2.1} and with $v_1(0)=\lambda$ are chosen.
Lemma \ref{c2.1} yields coordinates $t_2,...,t_{n+1}$
on $B_0$ and a prepotential $\Psi\in \OO_{B_0}$.
Then the formulas \eqref{3.21} and \eqref{3.22} in 
section \ref{c3.5} provide sections $\sigma_\alpha$
which generate the tautological family of (TEP)-structures
on $B_2\subset \check D_{BL}$.

\begin{lemma}
(a) For $i,j,k\in\{2,...,n+1\}$ and $a\in\{n+2,...,2n+1\}$
\begin{eqnarray}\nonumber
z\nnn\sigma_1 &=& 
\sum_i\sigma_i\ddd t_i + \sum_a \sigma_a\ddd y_a 
+ \sigma_{2n+2}\ddd y_{2n+2} \\
&&+\left(\sum_a y_a\sigma_a+2y_{2n+2}\sigma_{2n+2}\right)
\frac{\ddd z}{z},\label{3.32} \\
z\nnn\sigma_i &=&
\sum_j\left(\sum_k\paa_i\paa_j\paa_k\Psi
\cdot\sigma_{n+k}\right)\cdot \ddd t_j \nonumber\\
&&+\sigma_{2n+2}\cdot \ddd y_{n+i} 
+y_{n+i}\cdot\sigma_{2n+2}\cdot \frac{\ddd z}{z}
+\sigma_i\ddd z,\label{ 3.33} \\
z\nnn\sigma_a &=& \sigma_{2n+2}\cdot \ddd t_{a-n} 
+ 2\sigma_a\cdot \ddd z,\label{3.34}\\
z\nnn\sigma_{2n+2} &=& 3\sigma_{2n+2}\cdot \ddd z.
\label{3.35}
\end{eqnarray}

(b) The family of tautological (TEP)-structures has a 
pole of Poincar\'e rank 1 along $\nmmm$ and is therefore
a (TEP)-structure on $B_2$. Its bundle is denoted 
$H^{B_2} \to \C\times B_2$.

(c) The sections $\sigma_\alpha$ define an extension
to a (trTLEP)-structure on $B_0$.
\end{lemma}
\begin{proof}
(a) These formulas follow from \eqref{3.21}, from the formulas
in lemma \ref{c2.1}, from \eqref{3.26} - \eqref{3.28} 
and from derivating the sections $\sigma_\alpha$ with 
$\nnn_{\frac{\paa}{\paa y_a}}$ and $\nnn_{\frac{\paa}{\paa y_{2n+2}}}$.

(b) Obvious. 

(c) This follows from (a) and \eqref{3.29}.
\end{proof}

\begin{lemma}
The bundle $H\to\C\times M$ whose sheaf
is $\OO(H)=e^{y_1/z}\cdot pr_2^*\OO(H^{B_2})$
(where $pr_2:\C\times B_2\to B_2$ is the projection)
is a (TEP)-structure with 
\begin{eqnarray}\label{3.36}
z\nnn (e^{y_1/z}\sigma_\alpha) 
= e^{y_1/z}\cdot z\nnn(\sigma_\alpha) 
+ e^{y_1/z}\cdot \sigma_\alpha\ddd y_1 
-y_1\cdot e^{y_1/z}\cdot\frac{\ddd z}{z}.
\end{eqnarray}
It satisfies the isomorphy condition \eqref{3.10}.
The sections $e^{y_1/z}\cdot\sigma_\alpha$
define an extension to a (trTLEP)-structure.
\end{lemma}

\begin{proof}
\eqref{3.36} shows that the pole along $\nmmm$ is of
Poincar\'e rank 1. The pairing $P$ satisfies
\begin{eqnarray*}
P(e^{y_1/z}\cdot \sigma_\alpha,e^{y_1/(-z)}\cdot\sigma_\beta)
= P(\sigma_\alpha,\sigma_\beta)\in z^3\cdot\C,
\end{eqnarray*}
so it is the pairing of a (TEP)-structure.
By \eqref{3.32} - \eqref{3.36} the sections 
$e^{y_1/z}\cdot\sigma_\alpha$ define an extension to a
(trTLEP)-structure. The Higgs field of the (TEP)-structure
satisfies the isomorphy condition \eqref{3.10}
because of \eqref{3.32} and \eqref{3.36}.
\end{proof}

Now theorem \ref{c3.1} applies and gives a canonical 
F-manifold structure. \eqref{3.36} shows 
$C_{\frac{\paa}{\paa y_1}}=\id$, therefore
$e=\frac{\paa}{\paa y_1}$.
In the following calculation, $[.]$ denotes 
the restriction to $H_{|\nmmm}$,
\begin{eqnarray*}
&&C_{y_1\paa_1-\sum_ay_a\paa_a-2y_{2n+2}\paa_{2n+2}}
[e^{y_1/z}\cdot \sigma_1]\\
&=& y_1[e^{y_1/z}\cdot \sigma_1] 
- \sum_a y_a[e^{y_1/z}\cdot\sigma_a]
- 2y_{2n+2}[e^{y_1/z}\cdot\sigma_{2n+2}]\\
&=& -U[e^{y_1/z}\cdot\sigma_1] 
:= -[z\nnn_\zdz e^{y_1/z}\cdot\sigma_1].
\end{eqnarray*}
Because $[e^{y_1/z}\cdot \sigma_1]$ generates
$\OO(H_{|\nmmm})$ as a $T_M$-module, 
this is sufficient to see 
$C_{y_1\paa_1-\sum_ay_a\paa_a-2y_{2n+2}\paa_{2n+2}}
=-U$ and $y_1\frac{\paa}{\paa y_1}-(E_{BL})_{|B_2}=E$.
Part (a) of theorem \ref{c3.6} is proved.

\bigskip
It rests to prove part (b).
The choice of the sections $v_1,...,v_{2n+2}$ yields 
coordinates $(t_2,...,t_{n+1})$ on $B_0$,
coordinates $(t_1,...,t_{2n+2})$ on $M^{U,\lambda}$
and coordinates $(y_1,t_2,...,t_{n+1}y_{n+2},...,y_{2n+2})$
on $M$. 

Of course, the most natural isomorphism between $M^{U,\lambda}$
and $M$ is by identifying these coordinates.
At the end of the proofs of lemma \ref{c2.1} and theorem 
\ref{2.2} it was discussed how the coordinates 
$(t_1,...,t_{2n+2})$ change if $(v_1,...,v_{2n+2})$ are
changed, but $U_\bullet$ and $\lambda$ are fixed.

One sees easily from \eqref{3.22} that the coordinates
$(y_1,t_2,...,t_{n+1},y_{n+2},...,y_{2n+2})$ change
in the same way. Therefore the isomorphism 
$M^{U,\lambda}\cong M$ above is canonical.
Obviously it respects unit field and Euler field.

It also respects the multiplication. To see this,
one chooses the section $\xi:=[e^{y_1/z}\cdot \sigma_1]$
in $H|_{\nmmm}$ and observes that the isomorphism
$C_\bullet\xi:TM\to H|_{\nmmm}$ maps $e$ to $\xi$,
$\paa_i$ to $[e^{y_1/z}\cdot \sigma_i]$,
$\frac{\paa}{\paa y_a}$ to $[e^{y_1/z}\cdot \sigma_a]$
and $\frac{\paa}{\paa y_{2n+2}}$ to 
$[e^{y_1/z}\cdot \sigma_{2n+2}]$.
The Higgs field of the multiplication on $TM$ is mapped
to the Higgs field $C$ on $H|_{\nmmm}$. 
One can extract the Higgs field $C$ from 
\eqref{3.32} - \eqref{3.36}. Comparison with \eqref{2.21}
shows that the multiplications coincide.
This proves part (b) of theorem \ref{c3.6}.
\end{proof}

\begin{remark}
(i) The isomorphism $C_\bullet\xi :TM\to H_{|\nmmm}$ above
with $\xi:=[e^{y_1/z}\cdot \sigma_1]$ 
lifts to an isomorphism from the (trTLEP)-structure
on $\pi^*TM$ in section \ref{3.2} (i) to the 
(trTLEP)-structure on $M$ in the last lemma,
with global sections $[e^{y_1/z}\cdot \sigma_\alpha]$.

(ii) In the beginning of section \ref{c3.3} a standard 
construction of Frobenius manifolds from meromorphic connections 
was mentioned. It can be applied to the Frobenius manifolds in 
theorem \ref{c2.2} and theorem \ref{c3.6}. There it uses the 
(trTLEP)-structure with isomorphy condition 
constructed in the last lemma 
and the isomorphism $C_\bullet\xi :TM\to H|_{\nmmm}$
with $\xi$ as above.
\end{remark}

\section{Projective special (K\"ahler) geometry}\label{c5}
\setcounter{equation}{0}

\noindent
This section presents some aspects of projective special geometry
in a form which will make the comparison with Frobenius manifolds
easy.
It does not offer new results, and it neglects some aspects,
for example the role of the pairing and an induced hermitian
metric.
Because of that we put the ``K\"ahler'' in brackets.
Projective special (K\"ahler) geometry has a purely holomorphic part,
the special coordinates, which are related to Frobenius manifolds,
and a part involving the real structure, which is not related to 
Frobenius manifolds, but to $tt^*$-geometry \cite{HERT:2003}.
We will touch the latter part only in the last part 
\ref{c5.4} of this section.
More complete accounts, different aspects and motivation are provided
in \cite{FREE:1999,CORT:1998,ALEX-CORT-DEVC:2002,BERS-CECO-OOGU-VAFA:1994}.

\subsection{The setting and two period maps}\label{c5.1}
Let $(B_0,\mathcal{V},\nnn,S,F^\bullet)$ be a variation of Hodge like filtrations
with pairing of weight $w=3$ which satisfies the CY-condition
\eqref{1.7} \& \eqref{1.8},
with $n=\dim B_0$.
As before, $B_0$ is a small neighborhood of a base point $0\in B_0$.
When necessary, the size of $B_0$ will be decreased, so essentially
the germ $(B_0,0)$ is considered.

The most important manifold in this section is 
$B=F^3-\{\textup{zero section}\}$, together with the natural projection
$p:B\to B_0$. The fibers $F^3_b-\{0\}\cong\C^*$ come equipped
with a $\C^*$-action from the vector space structure,
the corresponding vector field on $B$ is denoted
$\varepsilon$.
Points in $B$ are denoted $(\delta,b)$ where $b\in B_0$ and 
$\delta\in F^3_b-\{0\}$.
The pull back with $p$ yields on $B$ a variation of Hodge like
filtrations with pairing $(p^*\mathcal{V},p^*\nnn,p^*S,p^*F^\bullet)$ of weight $3$.
The bundle $p^*\mathcal{V}$ carries the tautological generating section 
$\sigma_{taut}$ with $\sigma_{taut}(\delta,b)=\delta$. It satisfies
\begin{eqnarray}\label{5.1}
(p^*\nnn)_\varepsilon \sigma_{taut} =  \sigma_{taut}.
\end{eqnarray}
There are two natural and related period maps:
\begin{eqnarray}\label{5.2}
P_1:B&\to& \mathcal{V}_0,\\
(\delta,b)&\mapsto& \nnn\textup{-flat shift of }\delta\in F^3_b\subset \mathcal{V}_b
\textup{ to }\mathcal{V}_0,\nonumber\\
P_2:TB &\to& p^*F^2,\label{5.3}\\
X&\mapsto& (p^*\nnn)_X\sigma_{taut}.\nonumber
\end{eqnarray}
Here $TB=T^{1,0}B$ is the holomorphic tangent bundle.
Only in the last section \ref{c5.4} 
also $T^{0,1}B$, $T^\C B=T^{1,0}B\oplus T^{0,1}M$
and $T^\R B$ will be used.

Because $B_0$ is small and $B$ is a $\C^*$-bundle on $B_0$,
the flat connection $\nnn$ induces the trivialization 
$\tau_1:\mathcal{V}\stackrel{\cong}{\to} \mathcal{V}_0\times B_0$ of the vector bundle $\mathcal{V}$, 
and $p^*\nnn$ induces the trivialization 
$\tau_2:p^*\mathcal{V}\stackrel{\cong}{\to} \mathcal{V}_0\times B$
of the vector bundle $p^*\mathcal{V}$.
\begin{lemma} The period maps $P_1$ and $P_2$ satisfy the following properties:
\begin{itemize}
\item[(a)] $P_1$ and $P_2$ are related by
\begin{eqnarray}\label{5.4}
\tau_2\circ P_2 = (P_1)_*:TB\to
P_1^*T\mathcal{V}_0 = P_1^*(\mathcal{V}_0\times \mathcal{V}_0) = \mathcal{V}_0\times B.
\end{eqnarray}
\item[(b)] $P_1$ is an embedding.
\item[(c)] $P_2$ is an embedding and thus (locally in $B$) an isomorphism
of vector bundles.
\end{itemize}
\end{lemma}

\begin{proof}
(a) It follows from the definitions.

(b) The restriction of $P_1$ to $p^{-1}(0)=F^3_0-\{0\}\subset B$
is the tautological embedding $F^3_0-\{0\}\to F^3_0\subset \mathcal{V}_0$.
Because of this and because $B_0$ is small, for $P_1$ being an
embedding it is sufficient to show that its differential $(P_1)_*$
is injective at points $(\delta,0)\in p^{-1}(0)$. At such points
$(P_1)_*=P_2: T_{(\delta,0)}B\to F^2_0\subset \mathcal{V}_0$.
This map $P_2: T_{(\delta,0)}B\to F^2_0$ is an isomorphism because of
\eqref{5.1} and the CY-condition \eqref{1.7}\&\eqref{1.8}.

(c) Because $B_0$ is small and $B$ is a $\C^*$-bundle on $B_0$,
also this follows from the fact that the map 
$P_2: T_{(\delta,0)}B\to F^2_0$ is an isomorphism for any
$(\delta,0)\in p^{-1}(0)$.
\end{proof}

\subsection{Flat structure}\label{c5.2}
The same situation as in section \ref{c5.1} is considered.
Now additionally a $\nnn$-flat subbundle $U_1\subset \mathcal{V}$ of rank $n+1$
and with $S(U_1,U_1)=0$ is chosen. 
It is called {\it opposite subbundle} if
$F^2+U_1=\mathcal{V}$, equivalent: $F^2+U_1=F^2\oplus U_1$, also equivalent:
$F^2\cap U_1=\{\textup{zero section}\}$.
Because $B_0$ is small, these conditions are also equivalent to their
restrictions to the zero fiber $\mathcal{V}_0$.

\begin{lemma}
(a) The following three conditions are equivalent:
\begin{itemize}
\item[($i$)]
$U_1$ is an opposite subbundle.
\item[($ii$)]
The composition $pr_1\circ P_1:B\to \mathcal{V}_0\to \mathcal{V}_0/(U_1)_0$ is an embedding,
so locally an isomorphism.
Here $pr_1:\mathcal{V}_0\to \mathcal{V}_0/(U_1)_0$ is the projection.
\item[($iii$)]
The composition $pr_2:p^*F^2\to p^*\mathcal{V}\to p^*\mathcal{V}/p^*U_1$ of embedding and 
projection is an isomorphism of vector bundles.
Then also $pr_2\circ P_2:TB\to p^*\mathcal{V}/p^*U_1$ is an isomophism of 
vector bundles.
\end{itemize}

(b) Suppose that (i)--(iii) hold. The vector space structure on 
$\mathcal{V}_0/(U_1)_0$ induces a flat structure on $B$ with flat
and torsion free connection $\nnn^{U_1}$.
The flat connection $p^*\nnn$ on $p^*\mathcal{V}$ induces a flat connection
on the quotient bundle $p^*\mathcal{V}/p^*U_1$ because $p^*U_1$ is a flat subbundle,
and that connection induces via $(pr_2\circ P_2)^*$ a flat connection
$\nnn'$ on $TB$. Then $\nnn'=\nnn^{U_1}$.
%
\end{lemma}

\begin{proof}
(a) (i)$\iff$(iii) is trivial.
The flat connection on $p^*\mathcal{V}/p^*U_1$ which is induced from $p^*\nnn$
on $p^*\mathcal{V}$, yields the trivialization
$\tau_3:p^*\mathcal{V}/p^*U_1\stackrel{\cong}{\to} \mathcal{V}_0/(U_1)_0\times B$
of the vector bundle $p^*\mathcal{V}/p^*U_1$.
Then 
\begin{equation}\label{5.5}
\tau_3\circ (pr_2\circ P_2) = (pr_1\circ P_1)_* 
\end{equation}
as maps from $TB$ to $(pr_1\circ P_1)^* T(\mathcal{V}_0/(U_1)_0) = \mathcal{V}_0/(U_1)_0\times B$.

(ii) is equivalent to two conditions:
First, that the restriction of $pr_1\circ P_1$ to $p^{-1}(0)$,
which is just the map
$$pr_1\circ P_1:F^3_0-\{0\}\hookrightarrow F^3_0\hookrightarrow \mathcal{V}_0
\to \mathcal{V}_0/(U_1)_0,$$
is an embedding, and second that the differential $(pr_1\circ P_1)_*$
at points of $p^{-1}(0)$ is an isomorphism.
The first condition is equivalent to $F^3_0\cap (U_1)_0 =\{0\}$
which is part of (i), and the second condition is equivalent to
(iii) and thus to (i), because of \eqref{5.5}.

(b) This follows from \eqref{5.5}.
\end{proof}

\subsection{Special coordinates}\label{c5.3}
The same situation as in \ref{c5.1} is considered.
The flat structure $\nnn^{U_1}$ on $B$ from an opposite subbundle
can be enriched by an additional choice, which leads to certain
flat coordinates, the {\it special coordinates}.

Now $a_1,...,a_{n+1},b_1,...,b_{n+1}$ are $\nnn$-flat sections of $\mathcal{V}$
which form a symplectic basis everywhere.
Then $U_1=\langle b_1,...,b_{n+1}\rangle$ and
$V_1=\langle a_1,...,a_{n+1}\rangle$ are $\nnn$-flat subbundles of rank
$n+1$ and with $S(U_1,U_1)=0=S(V_1,V_1)$.

By abuse of notation we write also $a_i$ for $p^*a_i$ and $b_i$ for
$p^*b_i$. There are unique functions $z_i,w_i\in \OO_{B}$, $i=1,...,n+1$,
with 
\begin{eqnarray}\label{5.7}
\sigma_{taut} =\sum_{i=1}^{n+1}z_i\cdot a_i + \sum_{i=1}^{n+1}w_i\cdot b_i.
\end{eqnarray}

\begin{lemma}The following properties hold true:
\begin{itemize}
\item[(a)] $\varepsilon(z_i)=z_i, \varepsilon(w_i)=w_i$.
If $z_1,...,z_{n+1}$ are coordinates on $B$, they are $\nnn^{U_1}$-flat.
Then $\varepsilon =\sum_{i=1}^{n+1}z_i\paa_{z_i}$. Furthermore,
\begin{eqnarray}\label{5.8}
z_1,...,z_{n+1}\textup{ are coordinates on }B
&\iff & U_1\textup{ is an opposite subbundle,}\\
w_1,...,w_{n+1}\textup{ are coordinates on }B\label{5.9}
&\iff & V_1\textup{ is an opposite subbundle}.
\end{eqnarray}
If $U_1$ is an opposite subbundle then $z_1,...,z_{n+1}$ are called
special coordinates. If additionally $V_1$ is an opposite subbundle
then $w_1,...,w_{n+1}$ are called adjoint special coordinates.
\item[(b)] Suppose that $U_1$ is an opposite subbundle. Then there is a 
unique function
$\Psi^{U_1,V_1}\in \OO_{B}$ with
\begin{eqnarray}\label{5.10}
\frac{\paa \Psi^{U_1,V_1}}{\paa z_i}&=&w_i\quad \textup{ for }i=1,...,n+1,\\
\textup{and }\quad \varepsilon(\Psi^{U_1,V_1}) &=&2\cdot \Psi^{U_1,V_1}.
\label{5.11}
\end{eqnarray}
It depends only on $U_1$ and $V_1$, not on the symplectic basis.
It is called a prepotential.
\item[(c)] Suppose that $U_1$ is an opposite subbundle and that
$V_1^0\in \mathcal{L}(\mathcal{V},U_1)$ where
\[
\begin{split}
\mathcal{L}(\mathcal{V},U_1) :=&\big\{V_1\subset \mathcal{V}~|~\nnn\textup{-flat subbundle of rank }n+1,\\
&\hspace{3cm}S(V_1,V_1)=0, \mathcal{V}=U_1\oplus V_1\big\}.
\end{split}
\]
$\C[z_1,...,z_{n+1}]_2$ denotes the polynomials homogeneous of degree 2.
Then $\Psi^{U_1,V_1}\in \Psi^{U_1,V_1^0}+\C[z_1,...,z_{n+1}]_2$, and the map
\begin{eqnarray}\nonumber
\mathcal{L}(\VV,U_1)\to \Psi^{U_1,V_1^0}+\C[z_1,...,z_{n+1}]_2,\quad V_1\mapsto \Psi^{U_1,V_1},
\end{eqnarray}
is a bijection.
\item[(d)] The class $\Psi^{U_1,V_1}+\C[z_1,...,z_{n+1}]_2$ of prepotentials
in (c) is characterized by the third derivatives 
$XYZ\Psi^{U_1,V_1}$ where $X,Y,Z\in \bigoplus_{i=1}^{n+1}\C\cdot\paa_{z_i}$
are flat vector fields, and these third derivatives are given by
\begin{eqnarray}\label{5.13}
-S(\sigma_{taut},\nnn_X\nnn_Y\nnn_Z\sigma_{taut})
=XYZ\Psi^{U_1,V_1}.
\end{eqnarray}
This is a coordinate free characterization of this class of prepotentials.
The class of prepotentials depends only on the flat structure
$\nnn^{U_1}$ on $B$.
\end{itemize} 
\end{lemma}

\begin{proof}
(a) \eqref{5.1} gives $\varepsilon(z_i)=z_i$, $\varepsilon(w_i)=w_i$.
The map $pr_1\circ P_1$ from section \ref{c5.2} is now explicitly
\begin{eqnarray}\label{5.14} 
pr_1\circ P_1 :B\to \mathcal{V}_0/(U_1)_0,\quad
(\delta,b)\mapsto \sum_{i=1}^{n+1}z_i\cdot [a_i].
\end{eqnarray}
It is an embedding iff $z_1,...,z_{n+1}$ are coordinates on $B$.
Lemma \ref{c5.2} applies and gives \eqref{5.8}.

If $z_1,...,z_{n+1}$ are coordinates, they are $\nnn^{U_1}$-flat because
of \eqref{5.14}.
In that case $\varepsilon = \sum_{i=1}^{n+1}\varepsilon(z_i)\paa_{z_i}
=\sum_{i=1}^{n+1}z_i\paa_{z_i}$.
\eqref{5.9} is analogous to \eqref{5.8}.

(b) $\nnn_{\paa_{z_i}}\sigma_{taut} 
= a_i+\sum_{j=1}^{n+1}\frac{\paa w_j}{\paa z_i}\cdot b_j$ is a section
in $p^*F^2$, and $S(F^2,F^2)=0$, so 
\begin{eqnarray}\nonumber
0=S(\nnn_{\paa_{z_i}}\sigma_{taut},\nnn_{\paa_{z_j}}\sigma_{taut})
=\frac{\paa w_i}{\paa z_j}S(a_i,b_i)
+\frac{\paa w_j}{\paa z_i}S(b_j,a_j)
=\frac{\paa w_i}{\paa z_j}-\frac{\paa w_j}{\paa z_i}.
\end{eqnarray}
There exists a function $\Psi\in \OO_{B}$ with 
$\frac{\paa \Psi}{\paa z_i}=w_i$. It is unique up to addition of a constant.
It is claimed that there is exactly one function
$\Psi^{U_1,V_1}$ in this class with 
$\varepsilon(\Psi^{U_1,V_1})=2\cdot \Psi^{U_1,V_1}$.
Obviously there exists at most one  such function.
For the existence observe
\begin{eqnarray}\nonumber
\paa_{z_i}\varepsilon(\Psi) = [\paa_{z_i},\varepsilon](\Psi)+
\varepsilon\paa_{z_i}(\Psi) = \paa_{z_i}(\Psi)+\varepsilon(w_i)
=2w_i.
\end{eqnarray}
Therefore $\frac{1}{2}\varepsilon(\Psi)$ is also in the class.
Because of $\frac{1}{2}\varepsilon(\Psi)=\Psi+constant$,
$\varepsilon(\frac{1}{2}\varepsilon(\Psi))=\varepsilon(\Psi)$,
so $\frac{1}{2}\varepsilon(\Psi)$ is the desired function
$\Psi^{U_1,V_1}$.

For the independence of the symplectic basis, consider a symplectic
base change which fixes $U_1$ and $V_1$,
\begin{eqnarray}\nonumber
(a_1',...,a_{n+1}')&=&(a_1,...,a_{n+1})\cdot A,\\
(b_1',...,b_{n+1}')&=&(b_1,...,b_{n+1})\cdot (A^{tr})^{-1}
\quad \textup{ with }A\in GL(n+1,\C).\nonumber
\end{eqnarray}
Then $(z_1',...,z_{n+1}')=(z_1,...,z_{n+1})\cdot (A^{tr})^{-1}$,
$(w_1',...,w_{n+1}')=(w_1,...,w_{n+1})\cdot A$,
$(\paa_{z_1'},...,\paa_{z_{n+1}'}) 
=(\paa_{z_1},...,\paa_{z_{n+1}})\cdot A$, and thus
$(\paa_{z_1'},...,\paa_{z_{n+1}'})(\Psi)=(w_1',...,w_{n+1}')$,
so $\Psi'=\Psi$.
Therefore $\Psi^{U_1,V_1}$ depends only on $U_1$ and $V_1$, not on the 
symplectic basis.

(c) Suppose that $a_1,...a_{n+1},b_1,...,b_{n+1}$, 
$U_1=\langle b_1,...,b_{n+1}\rangle$ and 
$V_1=\langle a_1,...,a_{n+1}\rangle$ are given, with $U_1$ an opposite
subbundle.
For any $V_1'\in \mathcal{L}(\VV,U_1)$ there are unique $a_1',...,a_{n+1}'\in V_1'$
such that $a_1',...,a_{n+1}',b_1,...,b_{n+1}$ are a symplectic basis
and 
$$(a_1',...,a_{n+1}')=(a_1,...,a_{n+1}) + (b_1,...,b_{n+1})\cdot A.$$
Then $A=A^{tr}$.

The corresponding map
$$\{A\in M((n+1)\times (n+1),\C) |\ A=A^{tr}\}\to\VV$$
is a bijection. This and the following formulas give the claimed
1-1 correspondence,
\begin{eqnarray}\nonumber
\sigma_{taut}&=& \sum_{i=1}^{n+1}z_ia_i+\sum_{i=1}^{n+1}w_ib_i
=\sum_{i=1}^{n+1}z_ia_i'
+\sum_{i=1}^{n+1}(w_i-\sum_{j=1}^{n+1}A_{ji}z_j)b_i,\\
w_i'&=& w_i-\sum_{j=1}^{n+1}A_{ji}z_j,\nonumber\\
\Psi^{U_1,V_1'}&=& \Psi^{U_1,V_1}-\frac{1}{2}\sum_{i,j}A_{ij}z_iz_j.
\nonumber
\end{eqnarray}
(d) Derivation of 
$0=S(\sigma_{taut},\nnn_{\paa_{z_j}}\nnn_{\paa_{z_k}}\sigma_{taut})$ 
by $\paa_{z_i}$ gives
\[
\begin{split}
-S(\sigma_{taut},
\nnn_{\paa_{z_i}}\nnn_{\paa_{z_j}}\nnn_{\paa_{z_k}}\sigma_{taut})
&=S(\nnn_{\paa_{z_i}}\sigma_{taut},
\nnn_{\paa_{z_j}}\nnn_{\paa_{z_k}}\sigma_{taut})\\
&= S\left(a_i+\sum_{m=1}^{n+1}(\paa_{z_i}\paa_{z_m}\Psi)\cdot b_m,
\sum_{l=1}^{n+1}(\paa_{z_j}\paa_{z_k}\paa_{z_l}\Psi)\cdot b_l\right)\nonumber\\
&=\paa_{z_i}\paa_{z_j}\paa_{z_k}\Psi.
\end{split}
\]
This completes the proof.
\end{proof}

\subsection{Data involving the real structure}\label{c5.4}
Now let $(B_0,\mathcal{V},\nnn,\mathcal{V}_\R,S,F^\bullet)$ be a VPHS of weight $w=3$
which satisfies the CY-condition \eqref{1.7} \& \eqref{1.8}.
In the sections \ref{c5.1} to \ref{c5.3} we concentrated on one 
purely holomorphic aspect of projective special geometry,
the flat structure and special coordinates after choosing 
$U_1$ and $a_1,...,a_{n+1},b_1,...,b_{n+1}$.

For the sake of completeness here we discuss another datum, a connection
$\nnn^{psg}$ on $T^\C B$ which involves the real structure.
A third aspect, a hermitian pairing from the polarization will not be
discussed here.
In the following $TB=T^{1,0}B$, $T^\C B=T^{1,0}B\oplus T^{0,1}B$
and $T^\R B$ will be used.

The period map $P_2$ and the real structure $\mathcal{V}_\R$ induce an extended
period map
\begin{eqnarray}\label{5.16}
P_3:T^\C  &\to& p^*\mathcal{V},\\
P_3=P_2:T^{1,0}B &\to& p^*F^2, \nonumber\\
P_3=''\overline{P_2}'':T^{0,1}B &\to& \overline{p^*F^2},\quad
X\mapsto \overline{P_2(\overline X)}.\nonumber
\end{eqnarray}

\begin{lemma} For this period map we have:
\begin{itemize}
\item[(a)] $P_3$ is an isomorphism of $\C$-vector bundles.
It respects the real structures, i.e. it maps $T^\R B$ to $p^*\mathcal{V}_\R$.
\item[(b)] Let $\nnn^{psg}$ be the connection on $T^\C B$ induced
by $p^*\nnn$ via $P_3$. It is flat and thus gives $T^\C B$ the 
structure of a holomorphic vector bundle.
Of course, the subbundles $P_3^*(p^*F^3)$, $P_3^*(p^*F^2)=T^{1,0}B$
and $P_3^*(p^*F^1)$ of $T^\C B$ are holomorphic subbundles
with respect to this holomorphic structure.
The connection $\nnn^{psg}$ is torsion free.
\end{itemize}
\end{lemma}

\begin{proof}
Part (a) is obvious after lemma 4.1 (c), 
in part (b) only the torsion freeness of $\nnn^{psg}$ is nontrivial.
As it is classical and we will not use it, we leave the proof to the reader.
\end{proof}

\begin{remark}
Let $J:T^\R B\to T^\R B$ with $J^2=-\id$ give the complex structure
on $B$. The condition that $T^{1,0} B\subset T^\C B$ is a holomorphic
subbundle with respect to the holomorphic structure on $T^\C B$ from
$\nnn^{psg}$ is equivalent to
\begin{eqnarray}\label{5.17}
\left(\nnn^{psg}_X J\right)(Y) = \left(\nnn^{psg}_Y J\right)(X)
\quad\textup{ for } X,Y\in T^\C_{B}
\end{eqnarray}
\cite[Lemma 3.6]{HERT:2003}.
The condition \eqref{5.17} is often used as defining condition for 
affine special geometry.
Thus affine special geometry on a manifold $M$ means that there is a 
torsion free and flat connection which together with the (Hodge)
decomposition $T^\C M=T^{1,0}M\oplus T^{0,1}M$ and the real subbundle
$T^\R M$ yields a variation of Hodge structures of weight 1 on the complex
tangent bundle $T^\C M$ \cite[Proposition 3.7]{HERT:2003}.
Of course, in the present situation this holds,
projective special geometry includes affine special geometry.
The Hitchin system on the other hand exhibits the opposite behaviour: we will show that
the natural affine special geometry refines to a projective one.
\end{remark}

\section{Comparison}\label{c6}
\setcounter{equation}{0}

\noindent 
Let $(B_0,\mathcal{V},\nnn,S,F^\bullet)$ be a variation of Hodge like filtrations
with pairing of weight 3 which satisfies the CY-condition 
\eqref{1.7} \& \eqref{1.8}, with $n=\dim B_0$.
As always, $B_0$ is supposed to be small, a germ of a manifold
at a base point $0\in B_0$.

In section \ref{c2} \& \ref{c3} we discussed a manifold $M\supset B_0$
of dimension $2n+2$ and Frobenius manifold structures on it
depending on a choice $(U_\bullet, \lambda_0)$,
where $U_\bullet$ is an opposite filtration and $\lambda_0\in F^3_0-\{0\}$.

In section \ref{c5} we discussed a manifold $B$ of dimension $n+1$ 
which is a $\C^*$-bundle on $B_0$, and a holomorphic aspect of 
projective special geometry, a flat structure (and special coordinates)
depending on a choice of an opposite subbundle $U_1$.

Now the constructions and data will be compared.

\subsection{Choice of $U_0$ and $U_2$}\label{c6.1}
In the first lemma 
we start with $B$ and a choice of the subbundles $U_0$ and $U_2$
of an opposite filtration $U_\bullet$. 
In the second lemma $U_1$ will be added.

\begin{lemma}
Let $U_0$ and $U_2$ be flat subbundles of $\mathcal{V}$ with 
\begin{eqnarray}\label{6.1}
U_0=(U_2)^{\perp_S},\ U_2=(U_0)^{\perp_S},\ \rank U_0=1,\ 
\rank U_2=2n+1.
\end{eqnarray}
\begin{itemize}
\item[(a)] Then $F^3+U_2=\mathcal{V} \iff F^1+U_0=\mathcal{V}$.
\item[(b)] Suppose that $F^3+U_2=\mathcal{V}$. The flat connection on the quotient bundle
$\mathcal{V}/U_2$ and the isomorphism $F^3\hookrightarrow \mathcal{V}\to \mathcal{V}/U_2$ yield a 
flat connection on $F^3$ and a trivialization 
$\tau_4:F^3\to F^3_0\times B_0$.
This restricts to a trivialization
\begin{eqnarray}\label{6.2}
P^{U_2}: B\to (F^3_0-\{0\})\times B_0
\end{eqnarray}
of the $\C^*$-bundle $B=F^3-\{\textup{zero section}\}$.
\item[(c)] The additional choice $\lambda_0\in F^3_0-\{0\}$ distinguishes a hypersurface
$(P^{U_2})^{-1}(\{\lambda_0\}\times B_0)\cong B_0$ in $B$.
\end{itemize}
\end{lemma}

\begin{proof}
For (a) remark $(F^3\cap U_2)^{\perp_S}=(F^3)^{\perp_S}+U_2^{\perp_S}=F^1+U_0$ 
and $F^3+U_2=\mathcal{V}\iff F^3\cap U_2=\{\textup{zero section}\}$.
(b) and (c) are clear.
\end{proof}

\begin{lemma} Let $(B_0,\mathcal{V},\nnn,S,F^\bullet)$ be a variation of Hodge like filtrations
with pairing of weight 3 satisfying the CY-condition. 
\begin{itemize}
\item[(a)] The choice of $U_0$ and $U_2$ with \eqref{6.1} and 
$F^3+U_2=\mathcal{V}$  and the choice of an opposite subbundle $U_1$ are together
just the choice of an opposite filtration $U_\bullet$.
\item[(b)] Suppose that such a choice is made.
Then the hypersurfaces
$(P^{U_2})^{-1}(\{\lambda_0\}\times B_0)\subset B$, $\lambda_0\in F^3_0-\{0\}$,
are $\nnn^{U_1}$-flat hyperplanes of $B$,
and they all induce the same flat structure on $B_0$.
\end{itemize}
\end{lemma}

\begin{proof}
(a) is trivial. (b) By the embedding 
$pr_1\circ P_1:B\to \mathcal{V}_0/(U_1)_0$ in lemma \ref{c5.2} (a)(ii),
the fibration of $B$ by hyperplanes 
$(P^{U_2})^{-1}(\{\lambda_0\}\times B_0)$, $\lambda_0\in F^{3}_0-\{0\}$,
is mapped to the fibration of $\mathcal{V}_0/(U_1)_0$ by the affine
hyperplanes $[\lambda_0]+(U_2)_0/(U_1)_0$.
\end{proof}

\subsection{Flat structures and (pre)potentials}\label{c6.2}

In the proof of theorem \ref{c2.2} (a), the choice $(U_\bullet,\lambda_0)$
with $U_\bullet$ an opposite filtration and $\lambda_0\in F^3_0-\{0\}$
led to a Frobenius manifold structure on $M\supset B_0$
with potential $\Phi=\Psi+...$ as in \eqref{2.19} and 
$\Psi\in \OO_{B_0}$. 
The additional choice of $v_1^0,...,v_{2n+2}^0$ lead to flat coordinates
$t_1,...,t_{2n+2}$ with $t_2,...,t_{n+1}$ flat coordinates on $B_0\subset M$.

In lemma \ref{c5.3} the choice of an opposite subbundle $U_1$ and another
subbundle $V_1$ led to a prepotential $\Psi^{U_1,V_1}\in \OO_{B}$
and a flat structure on $B$. 
The additional choice of a symplectic basis $a_1,...,a_{n+1},b_1,...,b_{n+1}$
with $U_1=\langle b_1,...,b_{n+1}\rangle$ 
and $V_1=\langle a_1,...,a_{n+1}\rangle$
led to flat special coordinates $z_1,...,z_{n+1}$ on $B$.
These data will be compared now.

\begin{theorem}
Choose $(U_\bullet,\lambda_0)$ as above. Choose $v_1^0,...,v_{2n+2}^0$ as in 
lemma \ref{c2.1}, with $v^0_1=\lambda_0$. Choose $a_i=v_i^0$ ($i=1,...,n+1$)
and $b_i=v_{n+i}^0$ ($i=2,...,n+1$) and $b_1=-v_{2n+2}^0$.
Then $a_1,...,a_{n+1},b_1,...,b_{n+1}$ are a symplectic basis, and 
$V_1=F^2_0$.
\begin{itemize}
\item[(a)] Then
\begin{eqnarray}\label{6.3}
(P^{U_2})^{-1}(\{\lambda_0\}\times B_0)=\{z_1=1\}\subset B,
\end{eqnarray}
and $B_0$ is embedded into $B$ as this hyperplane.
The flat structure on $B_0$ from the Frobenius manifold coincides
with the flat structure which $B_0$ inherits from $B$ by this
embedding.
\item[(b)] The following equalities hold true:
\begin{eqnarray}\label{6.4}
t_i &=& z_i|_{\{z_1=1\}}\qquad\textup{ for }i=2,...,n+1,\\
\Psi &=& \Psi^{U_1,F^2_0}|_{\{z_1=1\}}\label{6.5}
\end{eqnarray}
\item[(c)] The potential $\Phi$ of the Frobenius manifold can be changed
by adding any element of $\C[t_1,...,t_{2n+2}]_{\leq 2}$
(where the index means degree $\leq 2$) without changing the 
Frobenius manifold.
\end{itemize}
All the prepotentials in the class $\Psi^{U_1,F^2_0}+\C[z_1,...,z_{n+1}]_2$
from lemma \ref{c5.3} (d) give via \eqref{6.5}
and \eqref{2.19} ($\Phi=\Psi+...$)  all the Frobenius manifold potentials
in the class $\Phi+\C[t_2,...,t_{n+1}]_{\leq 2}$.
\end{theorem}

\begin{proof}
Compare \eqref{2.3} and \eqref{5.7},
\begin{eqnarray}\label{6.6}
v_1&=& v_1^0+\sum_{2=1}^{n+1}t_i\cdot v_i^0 
+ \sum_{2=1}^{n+1}\paa_i \Psi\cdot v_{n+i}^0
+((\sum_{k=2}^{n+1}t_k\paa_k-2)\Psi)\cdot v_{2n+2}^0,\\
\sigma_{taut}&=& z_1\cdot v_1^0 + \sum_{i=2}^{n+1}z_i\cdot v_i^0
+\sum_{i=2}^{n+1}w_i\cdot v_{n+i}^0 - w_1\cdot v_{2n+2}^0.
\label{6.7}
\end{eqnarray}
On the hyperplane $\{z_1=1\}$ the section $\sigma_{taut}$
restricts to $v_1$, with 
\begin{eqnarray*}
t_i&=&z_{i|\{z_1=1\}},\\ 
\paa_i\Psi&=&w_{i|\{z_1=1\}}\quad \textup{ for }\quad i=2,...,n+1,\\ 
(\sum_{k=2}^{n+1}t_k\paa_k-2)\Psi&=&-w_{1|\{z_1=1\}}.
\end{eqnarray*}
The equations $t_i=z_{i|\{z_1=1\}}$ show part (a) and \eqref{6.4}.
The equations $\paa_i\Psi=w_{i|\{z_1=1\}}\textup{ for }i=2,...,n+1$ give
\begin{eqnarray}\nonumber
\frac{\paa}{\paa t_i}((\Psi^{U_1,F^2_0})_{|\{z_1=1\}}) = 
(\frac{\paa}{\paa z_i}\Psi^{U_1,F^2_0})_{|\{z_1=1\}} = w_{i|\{z_1=1\}}
=\frac{\paa}{\paa t_i}\Psi.
\end{eqnarray}
This shows $(\Psi^{U_1,F^2_0})_{|\{z_1=1\}}=\Psi+\textup{constant}$.
In order to see that this constant is 0, we use
$(\sum_{k=2}^{n+1}t_k\paa_k-2)\Psi=-w_{1|\{z_1=1\}}$ and $v_1(0)=v_1^0$,
which gives the first equality in the following equations,
\begin{eqnarray}\nonumber
0&=& -((\sum_{k=2}^{n+1}t_k\paa_k-2)\Psi)(0)
=w_1(z_1=1,z_i=0) \quad (i=2,...,n+1)\\
&=& \left(\frac{\paa\Psi^{U_1,F^2_0}}{\paa z_1}
\right)(z_1=1,z_i=0)\nonumber\\
&=& \left(\varepsilon \Psi^{U_1,F^2_0}\right)(z_1=1,z_i=0)\nonumber\\
&=&2\cdot \Psi^{U_1,F^2_0}(z_1=1,z_i=0).\nonumber
\end{eqnarray}
As $\Psi(0)=0$, this shows \eqref{6.5}.
Part (c) is clear.
\end{proof}

\begin{remark}
The theorem says that the Frobenius manifold structures on $M$
with choices $(U_\bullet,\lambda_0)$ with fixed $U_1$, but varying
$(U_0,U_2,\lambda_0)$ have a nice common geometric origin.
The flat structures on $B_0$ come from different embeddings of $B_0$
as affine hyperplanes in the flat manifold $B$.
The parts $\Psi$ of the Frobenius manifold potentials 
$\Phi=\Psi+...$ arise via restriction of the same prepotential 
$\Psi^{U_1,F^2_0}$.
\end{remark}

\input{hit-vhs}

\bibliographystyle{amsalpha}
\newcommand{\etalchar}[1]{$^{#1}$}
\def\cprime{$'$}
\providecommand{\bysame}{\leavevmode\hbox to3em{\hrulefill}\thinspace}
\providecommand{\MR}{\relax\ifhmode\unskip\space\fi MR }
\providecommand{\MRhref}[2]{%
  \href{http://www.ams.org/mathscinet-getitem?mr=#1}{#2}
}
\providecommand{\href}[2]{#2}

\end{document}

%% file: hit-vhs.tex
\section{Hitchin systems}
\label{hit}
The remainder of this paper is devoted to the application of the theory developed thus far to 
certain integrable systems as constructed in \cite{HITC:1987}.  
These are examples of so-called \textit{algebraically completely integrable systems}, 
which in turn are known to give variations of Hodge structures of weight one on their base space.
We will show that this can be refined in a natural way to a variation of Hodge like filtrations of weight three as described in
the first part of the paper, which allows us to apply the results formulated there.
We begin with a brief review of these integrable systems.

\subsection{The moduli space of Higgs bundles}
\label{mhb}
Let $C$ be a complex curve of 
genus $g(C)\geq 2$, and fix a complex reductive group $G$ with Lie algebra $\mathfrak{g}$.
A \textit{principal Higgs bundle} is a pair $(P,\Phi)$, where $P\to C$ is
a holomorphic principal $G$-bundle over $C$, and 
$\Phi$ --called the Higgs field-- is an element of $H^0(C,{\rm ad}(P)\otimes K_C)$, that is, 
a holomorphic one-form with values in the adjoint bundle  
${\rm ad}(P)$ of $P$.

Recall that a principal $G$-bundle $P$ is said to be  \textit{stable} if the adjoint
bundle is a stable vector bundle, i.e., for every
proper subbundle $F\subset {\rm ad}(P)$, 
we have $\deg(F)/\rk(F)<\deg ({\rm ad}(P))/\rk ({\rm ad}(P))$. 
As proved in \cite{RAMA:1975}, the moduli space $\mathcal{M}$ of 
stable principal $G$-bundles is a smooth quasi-projective complex variety
of dimension $\dim\mathcal{M}=\dim G(g(C)-1)+\dim Z(G)$, where 
$Z(G)$ is the center of $G$. Its tangent space is given by 
\[
T_{[P]}\mathcal{M}\cong H^1(C,{\rm ad}(P)),
\]
so by Serre-duality, a Higgs bundle whose underlying principal bundle
is stable determines a unique point in $T^*\mathcal{M}$. 

The complex manifold $\mathfrak{X}:=T^*\mathcal{M}$ forms an open
dense subspace of the full moduli space of Higgs bundles. As a cotangent bundle,
it carries a canonical holomorphic symplectic form $\omega_{can}$: the tangent space to
$\mathfrak{X}$ at the point $[P,\Phi]$  fits
into an exact sequence
\[
0\rightarrow H^0(C,{\rm ad} (P)\otimes K_C)\rightarrow T_{[P,\Phi]}\mathfrak{X}\rightarrow H^1(C,{\rm ad} (P))\rightarrow 0.
\]
The symplectic form is just the antisymmetrized version of the pairing between the first and third entry
as induced by Serre-duality. Alternatively, 
there is a gauge-theoretical construction of this moduli space \cite{HITC:1987} which also
explains the \textit{hyperk\"ahler} nature of $\mathfrak{X}$. We shall not be concerned 
in this paper with this enriched structure except for the existence of a K\"ahler form $\omega_K$
on $\mathfrak{X}$ which is of type $(1,1)$ with respect to the canonical complex structure as a 
cotangent bundle to a complex manifold.

We will now describe Hitchin's fibration 
\[
p:\mathfrak{X}\rightarrow {\tilde B}:=\bigoplus_{i=1}^kH^0(C,K_C^{\otimes d_i}),
\]
where $k={\rm rank}(\mathfrak{g})$.
Choose a basis of invariant polynomials $p_1,\ldots p_k\in\C[\mathfrak{g}]^{G}$,
where $p_i$ has degree $d_i$. Each of these $p_i$ defines a map 
\[
p_i:H^0(C,{\rm ad}(P) \otimes K_C)\rightarrow H^0(C,K_C^{\otimes d_i}).
\]
Now $p$ is simply induced by the map $p(P,\Phi):=\sum_{i=1}^kp_i(\Phi).$
The fundamental theorem of Hitchin \cite{HITC:1987} states that 
the map $p$ defines an algebraic integrable system on $\mathfrak{X}$. 
This means that
\begin{itemize}
\item[$i)$] $p$ is the restriction of a proper holomorphic map to an open dense subspace whose generic fibers are Lagrangian with respect to the holomorphic symplectic form $\omega_{can}$,
\item[$ii)$] the K\"ahler form $\omega_K$ restricts to each fiber to define a positive polarization. 
\end{itemize}
\subsection{Cameral curves and abelianization}
\label{abelianization}
Let $\Delta\subset {\tilde B}$ be the discriminant of the map $p$ above and define $B:={\tilde B}\backslash \Delta$. By Hitchin's result stated above,
the fiber $\mathfrak{X}_b:=p^{-1}(b)\subset\mathfrak{X}$ is a dense open subset of a compact polarized abelian variety of dimension $\dim G(g(C)-1)+\dim 
Z(G)$ for each $b\in B$. It can be identified as a generalized Prym variety of a branched
cover $C_b$ of $C$, called the \textit{cameral cover}.

Fix a maximal torus $T\subseteq G$ with Lie algebra $\mathfrak{t}\subseteq\mathfrak{g}$, 
a Borel subgroup $H$ of $G$ which contains $T$, and denote the associated Weyl group by $W$. 
By Chevalley's theorem, restriction of polynomials induces an isomorphism 
$\C[\mathfrak{g}]^G\cong\C[\mathfrak{t}]^W$. 
Consider now the quotient map $\mathfrak{t}\to\mathfrak{t}/W$. 
Twisted with the canonical bundle $K_C$ this defines a Galois covering
$\mathfrak{t}\otimes K_C\to (\mathfrak{t}\otimes 
K_C)/W$, and observe that $(\mathfrak{t}\otimes 
K_C)/W\cong\bigoplus_{i=1}^kK_C^{\otimes d_i}$.
With this, the cameral cover for $b\in B$ is defined as
\[
C_b:=b^*(\mathfrak{t}\otimes K_C)\subset \mathfrak{t}\otimes K_C.
\]
The projection of the bundle $\mathfrak{t}\otimes K_C$ to the base $C$ induces 
a projection $\pi_b:C_b\to C$. By construction, this defines a $W$-Galois covering
of $C$, where the Weyl group acts by the restriction of the action on $\mathfrak{t}$.
\begin{remark}
For the classical groups, it is sometimes more convenient to use the smaller 
\textit{spectral covers} which are associated to representations of $G$, or rather their highest weights. 
Let us explain this for the case $G=GL(n,\C)$ and the fundamental representation. 
In this case the underlying moduli space $\mathcal{M}$ is of course simply
the moduli space of stable vector bundles of rank $n$. Let $\lambda\in\Lambda$ 
be the weight of the fundamental representation of $GL(n,\C)$ on $\C^n$, and 
denote its stabilizer under the action of the Weyl group by $W_\lambda$.
The spectral cover is defined as the quotient $C_b\slash W_\lambda$. 
Typically however, spectral covers suffer from singularities and it is easier to use the cameral.
\end{remark}

The abelianization procedure is the following:
for any principal $G$-bundle $P$ over $C$, the structure group of the pull-back 
$\pi_b^*P$ has a canonical reduction to $H$. The $T$-bundle associated to the 
projection $H\to T$ may not be $W$-invariant, but choosing a theta-divisor on 
$C$ gives a canonical twist to a $W$-invariant $T$-bundle \cite{SCOG:1998}.
With this one proves:
\begin{theorem}[Abelianization, see \cite{DONA:1993,FALT:1993,HITC:1987,SCOG:1998}]
\label{ab-cam}
\hspace{2cm}
\begin{itemize}
\item[$i)$] Locally around a point $(P,\Phi)\in\mathfrak{X}_b$, the moduli space of Higgs bundles
$\mathfrak{X}$ is isomorphic to
the moduli space of pairs $(\tilde{C},\tilde{P})$, where $\tilde{C}$ is a $W$-invariant deformation 
of the cameral cover $C_b$, and $\tilde{P}$ is a $W$-invariant $T$-bundle over it. 
\item[$ii)$] With this isomorphism, the projection $(\tilde{C},\tilde{P})\to \tilde{C}$ defines a 
Lagrangian foliation of an open subset of $\mathfrak{X}$.
\end{itemize}
\end{theorem}
Weyl group invariant infinitesimal deformations of $C_b$ in $\mathfrak{t}\otimes K_C$ 
are given by elements in $H^0(C_b, N_{C_b})^W$, where $N_{C_b}\to C_b$ is the 
normal bundle to $C_b\hookrightarrow \mathfrak{t}\otimes K_C$. The symplectic form
on $K_C$ defines an isomorphism $N_{C_b}\cong \mathfrak{t}\otimes K_{C_b}$ so that $ii)$ above 
gives the exact sequence
\begin{equation}
\label{estms}
0\to H^1(C_b,\mathfrak{t}\otimes \mathcal{O}_{C_b})^W\to T_{(P,\Phi)}\mathfrak{X}\to H^0(C_b,\mathfrak{t}\otimes K_{C_b})^W\to 0.
\end{equation}
In view of the Hitchin map this gives an identification $T_bB \cong H^0(C_b,\mathfrak{t}\otimes K_{C_b})^W$.

\section{The Seiberg--Witten differential}
\label{SWdif}
In this section we will define the Seiberg--Witten differential on the cameral curves associated 
to the Hitchin system and study its properties. In particular, we will relate the differential 
to the $\C^*$-action on the moduli space of Higgs bundles.
\subsection{The $\C^*$-action}
Let $(P,\Phi)$ be a Higgs bundle over the curve $C$ with $P$ a stable $G$-bundle. For $\xi\in\C^*$, we 
can scale the Higgs field to $\xi\Phi$ to obtain another Higgs bundle and this induces a
holomorphic action $\varphi_\xi(P,\Phi):=(P,\xi\Phi)$ on the moduli space $\mathfrak{X}$. 
Of course, this is simply the canonical action of $\C^*$ on the cotangent bundle 
$T^*\mathcal{M}$, from which one immediately deduces that
\[
\varphi_\xi^*\omega_{can}=\xi\omega_{can},
\]
i.e., the canonical symplectic form is \textit{conformal} with respect to the $\C^*$-action.
Let $E$ be the generating (holomorphic) vector field of this action, and define the Liouville form as
$\alpha:=\iota_{E}\omega_{can}$. By the conformal property of the symplectic form 
above we have $\Lie_E\omega_{can}=\omega_{can}$ and therefore $d\alpha=\omega_{can}$.

Let $b\in B$ and consider the restriction $\alpha_b:=\alpha|_{p^{-1}(b)}$, 
a holomorphic one-form on the fiber $p^{-1}(b)$. Recall that Hitchin's result stated in section \ref{mhb}
identified this fiber as a dense open subset of an Abelian variety. 
\begin{lemma}
The holomorphic one-form $\alpha_b$ is translation invariant.
\end{lemma}
\begin{proof}
As above, let $(p_1,\ldots,p_k)$ denote the components of the Hitchin map $p:\mathfrak{X}\to B$. 
Standard symplectic geometry shows that the Hamiltonian vector fields $X_i$ of $p_i$ for $i=1,\ldots,k$
are tangential to the fibers of $p$ and precisely generate the affine symmetry the fiber $p^{-1}(b)$ exhibits as an 
Abelian variety. Let $i_b:p^{-1}(b)\hookrightarrow\mathfrak{X}$ be the canonical inclusion.
Then we have
\[
\begin{split}
\Lie_{X_i}\alpha_b&=(d\iota_{X_i}+\iota_{X_i}d)i^*_b\alpha\\
&=i_b^*d\iota_{X_i}\iota_E\omega_{can}+\iota_{X_i}i^*_b d\alpha\\
&=-i_b^*d\iota_Edp_i+\iota_{X_i}i^*_b\omega_{can}\\
&=-d_ii_b^*(dp_i)\\
&=0.
\end{split}
\]
Here we have used that the fibration $p:\mathfrak{X}\to B$ is Lagrangian, i.e.,
$i_b^*\omega_{can}=0$ and that the $p_i$ are homogeneous of degree $d_i$.
\end{proof}
Introduce the following $\C^*$-action on the base $B$ of the Hitchin system: 
\[
\xi \cdot(b_1,\ldots,b_k)=(\xi^{d_1}b_1,\ldots,\xi^{d_k}b_k),
\]
where $\xi\in\C^*$ and $b=(b_1,\ldots,b_k)\in B$ with 
$b_i\in H^0(C,K_C^{\otimes d_i})$. Obviously, equipped with this action, 
the Hitchin map $p:\mathfrak{X}\to B$ is $\C^*$-equivariant. In the following, we denote 
the generating vector field of this action on $B$ by $\mathcal{E}$.

\subsection{Definition and properties}
A translation invariant one-form on an Abelian variety determines a unique
element in the linear dual of the tangent space at a generic point. Consulting
the short exact sequence \eqref{estms}, this means an element in $H^0
(C_b,\mathfrak{t}\otimes K_{C_b})^W$ for the case at hand, viz. the fiber 
$\mathfrak{X}_b:=p^{-1}(b)$ of the Hitchin map:
\begin{definition}
The Seiberg--Witten differential on the cameral curve $\lambda_{SW}\in H^0
(C_b,\mathfrak{t}\otimes K_{C_b})^W$ is the holomorphic 
one-form determined by the translation invariant one-form $\alpha_b$, the 
restriction of the Liouville form to the fiber $\mathfrak{X}_b$.
\end{definition}
There is an alternative definition of this differential as follows: Recall
that the cameral curve $C_b$ is canonically embedded in the total space 
of the vector bundle $\mathfrak{t}\otimes K_C$. There is a holomorphic action
of $\C^*$ by scaling along the fibers of this bundle. As a holomorphic cotangent
bundle, $K_C$ carries a canonical holomorphic symplectic form. On the tensor 
product $\mathfrak{t}\otimes K_C$, this can be interpreted as an $\mathfrak{t}$-valued 
symplectic form, denoted $\omega_{K_C}$. Let $\partial_\xi$ be the generator of the $\C^*$-action. 
Once again, the contraction $\theta:=\iota_{\partial_\xi}\omega_{K_C}$, called the Liouville form, is a 
potential for this symplectic form. 
\begin{proposition}
The Seiberg--Witten form is equal to the restriction of the Liouville form:
\[
\lambda_{SW}=\theta|_{C_{b}} 
\]
\end{proposition}
\begin{proof}
This is a consequence of the abelianization of Higgs bundles as described in section \ref{abelianization}.
Recall that the Hitchin map $p:\mathfrak{X}\to B$ is $\C^*$-equivariant, and projects the generating
 vector field $E$ to $\mathcal{E}$. Let $(P,\Phi)\in\mathfrak{X}_b$. Because the Hitchin map defines 
a Lagrangian fibration, and the Seiberg--Witten differential $\lambda_{SW}$ is defined by restricting
$\iota_E\omega_{can}$ to the fiber $\mathfrak{X}_b$ over $b\in B$, it follows from the exact
sequence \eqref{estms} that it is given by
\[
\lambda_{SW}=\mathcal{E}(b)\in  T_b B=H^0(C_b,\mathfrak{t}\otimes K_{C_b})^W.
\]
Let $\xi\in \mathbb{C}^*$. It is an easy consequence of the definitions that 
\[
C_{\xi \cdot b}=\xi \cdot \{C_b\},
\]
where on the right hand side we use the canonical $\C^*$-action on $K_C$ and the embedding 
$C_b\hookrightarrow \mathfrak{t}\otimes K_{C_b}$. The generator $\partial_\xi$ of this action 
therefore defines a $W$-invariant deformation of $C_b$ in $\mathfrak{t}\otimes K_{C_b}$ which 
corresponds to $\mathcal{E}$ using the isomorphism $H^0(C_b,N_{C_b})^W\cong H^0
(C_b,\mathfrak{t}\otimes K_{C_b})^W$.  As explained below Theorem \ref{ab-cam}, this 
isomorphism is induced by contracting with the symplectic form on $K_C$. But the Liouville form is 
precisely defined as  $\iota_{\partial_\xi}\omega_{K_C}$, so the result now follows.
\end{proof}

Some of the information about the cameral cover is conveniently encoded in the zero divisor $D_{\lambda_{SW}}$ of $\lambda_{SW}$.
The previous proposition clarifies where these zeroes are: using the fact that $\omega_{K_C}$ is nondegenerate one finds for any vector field $v$ that $\iota_v \lambda_{SW}(p)=0$ for $p\in C_b$ if and only if $\theta(p)=0$ or $v(p)=c\cdot \partial_\xi(p)$ for some constant $c$. The first set of points are the intersections of $C_b$ with $C$ while the second set consists of the branch points of the covering map $\pi_b:C_b \to C$. We split $D_{\lambda_{SW}}=D_{int} + D_{br}$ into the intersection and branch points accordingly and calculate their degrees, cf.\ \cite{KSIR:2001}. The map $\pi_b$ has degree $|W|$, the order of the Weyl group. 
By definition, $\deg(D_{int})=\deg(C_b \cap s_0)$ with $s_0$ the zero section of $K_C \to C$. This is the same as the intersection degree with any other section $s\in H^0(C,K_C)$
\[
\deg(D_{int}) = \deg(C_b \cap s) = \deg(\pi_b)\deg(s) = |W|\cdot |K_C|
\]
We now turn our attention to the branch points.
Since the cameral cover is the pull-back via $b$ of the $W$-Galois cover ${\mathfrak{t}} \to {\mathfrak{t}}/W$, we are interested in the branch points of the latter. If $\sigma_\alpha$ denotes the reflection in the root $\alpha$, then 
\[
\sigma_\alpha h =h \quad  \Leftrightarrow \quad \alpha(h)=0
\]
The map ${\mathfrak{t}}\mapsto{\mathfrak{t}}/W$ has branch points exactly on the zero divisor of the map $h\to \prod_\alpha \alpha(h)$.
This gives a degree $\Delta$ hypersurface $H\subset {\mathfrak{t}} \otimes K_C$, where $\Delta$ denotes the number of roots of $\mathfrak{g}$. The branch divisor of the cameral cover is the intersection divisor of $b$ with $H$ and therefore has degree $|\Delta | \cdot |K_C|$ where $|K_C|=2g(C)-2$ is the degree of the 
canonical divisor. 
This immediately gives
\[
\deg(D_{br}) = |\Delta | \cdot |K_C| 
\]
This is consistent with the Riemann-Hurwitz formula, which in this case reads 
\[
g(C_b) =\frac{|W|\cdot |K_C|}{2}+\frac{|K_C|\cdot |\Delta|}{2}+1 
\]
so that indeed
\[
\deg(D_{\lambda_{SW}}) =2g(C_b)-2 =  |K_C|(|W|+|\Delta |) = \deg(D_{int})+\deg(D_{br}).
\]
The multiplicities of the points in $D_{\lambda_{SW}}$ may depend on the point in the base $B$, but in the generic situation $C_b \cap s_0$ consists of transversal intersections (giving first order zeroes of $\lambda_{SW}$) and the branch points are all of second order (giving second order zeroes). From now on, we will assume to be in the generic situation.

\section{Variations of Hodge structures from Cameral curves}
\label{camVHS}
In this section we study a variation of Hodge structures associated to the family of cameral curves
of the Hitchin system. A priori, this is a variation of weight one; in physics terminology the base 
is a \textit{rigid} special K\"ahler manifold, in mathematical terms it is called \textit{affine} special K\"ahler (cf. \cite{FREE:1999, ALEX-CORT-DEVC:2002, HERT:2003}). 
However, a careful analysis of the Seiberg--Witten differential in this variation shows that there exists a canonical refinement to a variation of weight three. 
In the physics literature this is called a \textit{local} special K\"ahler manifold, in the mathematics literature on refers to this situation as \textit{projective} special K\"ahler.

\subsection{The variation of weight one}
\label{var-1}
We review the variation of Hodge structures of weight $w=1$ over $B$ using the setup as in \cite{DELI:1970}.
The family of cameral curves $f:\mathscr{C}\to B$ is defined such that $\mathscr{C}_b:=f^{-1}(b)\cong C_b$.
Recall that $\mathscr{C}$ is equipped with an action of the Weyl group which preserves
the fibers of $f$. Consider now the direct image functor of $f$ in 
the category of $W$-equivariant  sheaves
\[
f_*:\mathsf{Sh}_W(\mathscr{C})\to\mathsf{Sh}(B),
\] 
which assigns to $\mathscr{S}\in\mathsf{Sh}_W(\mathscr{C})$ the sheaf
\[
U\mapsto \mathscr{S}(f^{-1}(U))^W. 
\]
Its derived functors are denoted by
$R^\bullet f_*$.
Let $\Lambda$ be the root lattice of $G$ and denote by $\underline \Lambda$ the associated locally constant sheaf on $\mathscr{C}$ equipped with the canonical $W$-action.
Homotopy invariance of cohomology implies that the sheaf of $\Z$-modules
\[
\mathcal{V}_\Z:=R^1f_*{\underline \Lambda}\in\mathsf{Sh}(B)
\]
forms a local system on $B$ whose stalk at $b\in B$ equals 
$(\mathcal{V}_\Z)_b=H^1(C_b,{\underline \Lambda})^W$. Next we consider the tensor product
\[
\mathcal{V}:=\mathcal{V}_\Z\otimes_\Z\mathcal{O}_{B},
\]
a coherent sheaf of holomorphic sections of a vector bundle over $B$.
Because ${\underline \Lambda}\otimes_\Z\C\cong\mathfrak{t}$, its fiber at $b\in B$ is given by $\mathcal{V}_b=H^1(C_b,\mathfrak{t})^W$. Obviously, the map $f$ is proper and therefore we have isomorphisms 
\[
\mathcal{V}\cong R^1f_*(\mathfrak{t}\otimes_\C f^* \mathcal{O}_{B})\cong \mathscr{H}^1 \left( f_* \left( \mathfrak{t} \otimes \Omega^\bullet_{\mathscr{C}\slash B} \right) \right)
\]
Here the relative differentials are defined through the following short exact sequence of coherent sheaves on ${\mathscr{C}}$
\begin{equation}
\label{rdf}
0\rightarrow f^* \Omega^\bullet_{B} \rightarrow\Omega^\bullet_{\mathscr{C}}
\rightarrow\Omega^\bullet_{\mathscr{C}\slash B}\rightarrow 0.
\end{equation}
The middle term carries a natural decreasing filtration via
\[
{\bf{F}}^k = \mbox{image} \left[ f^*\Omega^k_{B} \otimes_{{\mathcal{O}}_{\mathscr{C}}} \Omega^{\bullet-k}_{{\mathscr{C}}} \to \Omega^{\bullet}_{{\mathscr{C}}}  \right]
\]
The associated spectral sequence degenerates and leads
to a filtration on $(\mathfrak{t}\otimes\Omega^\bullet_{\mathscr{C}/B},d)$, the Hodge filtration. For the case at hand, this filtration has weight one; 
$F^1\subset{F}^0=\mathcal{V}$, with ${F}^1=f_*\left( \mathfrak{t}\otimes \Omega^1_{\mathscr{C}\slash B}\right)$, i.e., ${F}^1_b=H^0(C_b,\mathfrak{t}\otimes 
K_{C_b})^W\subset H^1(C_b,\mathfrak{t})^W.$ The differential 
\[
\nabla: E^{0,1}_1 \cong \mathcal{V} \to E_1^{1,1}\cong f^* \Omega^1_{B}\otimes \mathcal{V}
\]
is a flat connection on $\mathcal{V}$, 
called the Gauss-Manin connection,  whose flat sections are given by $\mathcal{V}_\mathbb{Z}\otimes_\Z\C$.
Finally, there is a polarization $S:\mathcal{V}\times\mathcal{V}\to\mathcal{O}_{B}$ given by
\begin{equation}
S_b(\alpha,\beta)=\left<\alpha\cup\beta,[C_b]\right>,
\end{equation}
where the cup-product includes taking the inner product of two elements in $\mathfrak{t}$.
Since we work with the first derived functor, it is antisymmetric: $S(\alpha,\beta)=-S(\beta,\alpha)$.
Furthermore, it is $\nabla$-flat:
\begin{equation} 
\label{pol-flat}
dS(\alpha,\beta)=S(\nabla\alpha,\beta)+S(\alpha,\nabla\beta).
\end{equation}
The total of these data $(B,\mathcal{V},\nabla,\mathcal{V}_\Z,S,{F}^\bullet)$  
define a variation of polarized Hodge structures of weight $w=1$, cf section \ref{c1.1}.

\subsection{The derivative of the Seiberg--Witten differential}
\label{section cdr}
Consider the variation of polarized Hodge structures $(B,\mathcal{V},\nabla,\mathcal{V}_\Z,S,{F}^\bullet)$ of weight $1$ associated to the family of cameral curves $f:\mathscr{C}\to B$ constructed in the previous section. 
By definition, the universal curve $\mathscr{C}$ comes equipped with an embedding
$\mathscr{C}\hookrightarrow \mathfrak{t}\otimes K_C\times B$. Pulling back  
the $\mathfrak{t}$-valued Liouville form $\theta$ on $\mathfrak{t}\otimes K_C$, one obtains a 
holomorphic one-form $\lambda$ on $\mathscr{C}$ which restricts to the Seiberg--Witten differential $\lambda_{SW}$ 
on each fiber $C_b$. In the following we write $\lambda_b$ for this restriction.  
By definition of the relative differential forms, the one-form $\lambda_{SW}$ defines a section of $\mathfrak{t}\otimes\Omega^1_{\mathscr{C}}$ which, under the projection to 
$\mathfrak{t}\otimes\Omega^1_{\mathscr{C}\slash B}$ and the direct image $f_*$, defines a section
$\lambda_{SW}\in {F}^1\subset\mathcal{V}$
and restricts to the Seiberg--Witten differential on each fiber:
\[
\lambda_b\in {F}^1_b=H^0(C_b,\mathfrak{t}\otimes K_{C_b})^W.
\]

\subsubsection{The \v{C}ech-de Rham resolution}
To compute the derivative of the Seiberg--Witten differential under the Gauss--Manin connection,
we use a \v{C}ech-resolution of the relative de Rham complex 
$(\Omega^\bullet_{\mathscr{C}/B},d)$ and calculate the hypercohomology following \cite{DELI:1970}. 
Define
\[
U=\{x\in\mathscr{C},~d\pi_{f(x)}\neq 0\},
\]
i.e., the complement of the branch points of the cameral cover or equivalently the complement of the second order zeroes of the Seiberg--Witten differential.
We choose $V\subset\mathscr{C}$ such that $V\cap C_b$ consists of a disjoint union of 
small disks $V_1,\dots,V_{|D_{br}|}$ around the second order zeroes $p_1,\dots,p_{|D_{br}|} \in C_b$ of $\lambda_{b}$.
Here ${|D_{br}|}=|\Delta||K_C|$ denotes the number of branch points, i.e., second order zeroes of $\lambda_{SW}$.
For any $W$-equivariant sheaf $\mathscr{S}\in \mathsf{Sh}_W(\mathscr{C})$, write
$f^U_*\mathscr{S}\in\mathsf{Sh}(B)$ short for the composition $f_*(i_{U})_*\mathscr{S}|_U$, 
where $i_U:U \hookrightarrow \mathscr{C}$ is the inclusion, and similarly for $f_*^V$ and $f^{U\cap V}_*$.

To compute $R^1f_*$ we need the following
part of the double complex of coherent sheaves on $B$:
\begin{equation}
\label{crd}
\xymatrix{f_*\Omega^1_{\mathscr{C}\slash B}(\mathscr{C})\ar[r]& 
f^U_*\Omega^1_{\mathscr{C}/B}\ar@(lu,ru)[]^{(L_{X_U},L_{X_V})} \oplus f^V_*\Omega^1_{\mathscr{C}\slash B}
\ar[r]^{\hspace{0.7cm}\delta} \ar[rd]^{\iota_{X_V-X_U}}  &f_*^{U\cap V}\Omega^1_{\mathscr{C}\slash B} \\
&f^U_*\Omega^0_{\mathscr{C}/B}\oplus f^V_*\Omega^0_{\mathscr{C}\slash B}\ar[r]^{\hspace{0.7cm}\delta} \ar[u]^{d_{\mathscr{C}\slash B}}&
f^{U\cap V}_*\Omega^0_{\mathscr{C}/B}\ar[u]^{d_{\mathscr{C}\slash B}}\ar@(u,ur)[]^{L_{X_U}}} 
\end{equation}
The vertical map $d_{\mathscr{C}\slash B}$ is the relative de Rham differential and $\delta$ denotes the \v{C}ech differential.
The notation $X_U,X_V$ will be explained below.
With this resolution, elements in $R^1f_*$ will  be represented as cocycles in 
\begin{equation}
\label{cdr}
\underbrace{\left(\Omega^1_{\mathscr{C}/B}(U)\oplus\Omega^1_{\mathscr{C}/B}(V)\right)}_{({\rm 
I})}\oplus\underbrace{\Omega^0_{\mathscr{C}/B}(U\cap V)}_{({\rm II})}
\end{equation}
so a relative differential $\alpha$ is represented by a triple
\[
(\alpha_U,\alpha_V,g_\alpha)
\]
satisfying
\[
d_{{\mathscr{C}}\slash B} g_\alpha = \delta \left(\alpha_V, \alpha_U\right)
\]
In terms of this complex, the Hodge filtration is given by $(\ref{cdr})$ and the polarization $S$ is given by a trace-residue pairing
\begin{equation}
\label{trres}
S_b(\alpha,\beta) = \sum_{k=1}^{|D_{br}|} Res_{V_k}  \left\langle g_\alpha dg_\beta \right\rangle
\end{equation}
where $\left\langle \hdots \right\rangle$ indicates the use of a pairing on $\mathfrak{t}$.

We now describe the Gauss-Manin connection. 
Over $U$ and $V$, one can choose splittings of the exact sequence of sheaves
\begin{equation}
\label{tangent}
0 \to f_*\Theta_{\mathscr{C}/B} \to f_* \Theta_\mathscr{C} \to  \Theta_B \to 0.
\end{equation}
This provides lifts $X_U,X_V$ of holomorphic vector fields $X$ on $B$.
Conversely, such lifts define a splitting.
The Gauss-Manin connection now has an explicit description: 
\begin{equation}
\label{GM}
\nabla_X (\alpha_U,\alpha_V,g_\alpha) = \left(L_{X_U}\alpha_U,L_{X_V}\alpha_V, L_{X_U}g_\alpha +\iota_{X_U-X_V}\alpha_V  \right)
\end{equation}
%
%
%
%
%
%
It is well-known that the part of the Gauss--Manin 
connection that actually shifts degree in the Hodge filtration, i.e., the Higgs field
\[
C_X : {{F}}^1 \to \mathcal{V} \slash {F}^1,
\]
equals taking the cup-product with the 
Kodaira--Spencer class: $C_X(\alpha)=\alpha\cup\kappa(X)$. Here 
$\alpha\in H^0(C_b,\mathfrak{t}\otimes K_{C_b})^W$, 
$\kappa:T_b B\to H^1(C_b,\Theta_{C_b})$ is the Kodaira--Spencer map and the notation stands short for the natural pairing
\[
H^0(C_b,\mathfrak{t}\otimes K_{C_b})^W\times H^1(C_b,\Theta_{C_b})\to H^1(C_b,\mathfrak{t}\otimes \mathcal{O}_{C_b})^W.
\]
In the relative \v{C}ech--de Rham complex, if $\alpha$ is represented by $(\alpha_U,\alpha_V,0)$ then $C_X$ is given by the interior product $\iota_{X_U-X_V}\alpha_V$, which maps $({\rm I})$ to $({\rm II})$ in \eqref{cdr}. 
\subsubsection{Derivatives of $\lambda_{SW}$} We now have the machinery to start the computation:
\begin{lemma}
\label{gt-deg3}
$\nabla_X\lambda_{SW}\in F^1_b$ for all $X\in T_b B$.
\end{lemma}
\begin{proof}
We have to show that the composition
\[
C_X:{F}^1\stackrel{\nabla_X}{\longrightarrow}\mathcal{V}\longrightarrow\mathcal{V}\slash{F}^1,
\]
applied to $\lambda_{SW}$, is zero. 
Since $\lambda$ is naturally defined on $\mathscr{C}$, $\lambda_{SW}$ can be represented as a differential on $\mathscr{C}$ by $(\lambda|_U,\lambda|_V,0)$.
One finds on $U\cap V$ that
\[
\iota_{X_U-X_V}\lambda|_V = \delta \left( \iota_{X_U}\lambda|_U, \iota_{X_V}\lambda|_V \right).
\]
Since this is exact, it follows that $C_X\lambda_{SW}=0$.
\end{proof}
Recall, cf. \eqref{estms}, that $T_b B\cong H^0(C_b,\mathfrak{t}\otimes K_{C_b})^W$. For $X\in T_b B$, we write
$\alpha_X$ for the holomorphic differential associated to $X$ by this isomorphism.
\begin{proposition}
\label{cy-hs}
For all $X\in T_b B$, we have 
\[
\nabla_X\lambda_{SW}=\alpha_X.
\]
\end{proposition}
\begin{proof}
We have already seen in the previous Lemma that the part of $\nabla_X\lambda$ 
which maps from $({\rm I})$ to $({\rm II})$ in the \v{C}ech--de Rham complex \eqref{cdr}, is exact. 
The remaining part, mapping $({\rm I})$ to $({\rm I})$, is given by taking the Lie derivatives 
$L_{X_U}$, $L_{X_V}$ with respect to holomorphic lifts of $X$ to $U$ and $V$.
By Cartan's formula
\[
L_{X_U}\lambda|_U=(d\iota_{X_U}+\iota_{X_U}d)\lambda|_U.
\] 
From this we see that
\[
\left(L_{X_U}\lambda_U,L_{X_V}\lambda_V\right)-d\left(\iota_{X_U}\lambda_U,\iota_{X_V}\lambda_V\right)=\left(\iota_{X_U}\omega_{K_C},
\iota_{X_V}\omega_{K_C}\right),
\]
where we have used that $\iota_{X_U}d\lambda|_U=\iota_{X_U}\omega_{K_C}|_{U}\in f_*^U \Omega^1_{\mathscr{C}\slash B}$.
Recall that the second term on the left hand side is exactly the derivative of the cocycle needed in the proof of the previous Lemma to make $C_X\lambda_{SW}$ equal to
 zero.
We now claim that on $U\cap V$ we have
\[
\iota_{X_U}\omega_{K_C}-\iota_{X_V}\omega_{K_C}=0
\]
in $f_*^{U\cap V}\Omega^1_{\mathscr{C}\slash B}$.
Indeed, the difference $X_U-X_V$ is a section of $\ker f_*\subset \Theta_{\mathscr{C}}$ and therefore tangent to each fiber $C_b$ of $f:\mathscr{C}\to B$. But $\omega_{K_C}$ is a ($\mathfrak{t}$-valued)
symplectic form, so $\iota_{X_U-X_V}\omega_{K_C}=0$ as a relative differential form. 
It follows that $\iota_{X_U}\omega_{K_C}|_{\mathscr{C}}$ and $\iota_{X_V}\omega_{K_C}|_{\mathscr{C}}$ are the
restrictions of an element of $f_*\Omega^1_{\mathscr{C}\slash B}$ which is by definition $\alpha_X$. 
\end{proof}
As a corollary one finds the following rather obvious fact:
\begin{corollary} For the generator $\mathcal{E}$ of the $\C^*$-action on $B$, we have:
\[
\nabla_\mathcal{E}\lambda_{SW}=\lambda_{SW}.
\]
\end{corollary}

\subsection{The variation of weight three}
Consider the variation of Hodge structures of weight one constructed in section \ref{var-1}. 
With the results of the previous section, we can now refine the filtration to a obtain a variation of Hodge like filtrations of weight $3$:
introduce
\[
\begin{split}
{\mathcal F}^3&:=\mathcal{O}_{B}\cdot \lambda_{SW}\\
{\mathcal F}^2&:=R^1f_*\Omega^1_{\mathscr{C}\slash B}\\
{\mathcal F}^1&:=\left({\mathcal F}^3\right)^{\perp S}\\
{\mathcal F}^0&:=\mathcal{V},
\end{split}
\]
and note that ${\mathcal F}^2=F^1$.
We introduce the projectivization $\mathbf{P}(B)$ with respect to the $\mathbb{C}^*$-action and obtain
\begin{theorem}
\label{var-weight3}
The data $(\mathbf{P}(B),\mathcal{V},\nabla, \mathcal{V}_\Z,S,{\mathcal F}^\bullet)$ define a variation of Hodge like filtrations of
 weight 3 satisfying the CY-condition.
\end{theorem}
\begin{proof}
Clearly, ${\mathcal F}_b^\bullet$ defines a decreasing filtration of weight $3$ on the fiber $\mathcal{V}_b$ 
over $b\in B$. 
Therefore, the only thing left to check is that the filtration satisfies Griffiths transversality with respect
to the Gauss--Manin connection, i.e.,
\[
\nabla{\mathcal F}^\bullet\subseteq{\mathcal F}^{\bullet-1}.
\]
In degree $3$, this property is equivalent to Lemma \ref{gt-deg3}. 
In degree $2$, let $\alpha\in f_*\Omega^1_{\mathscr{C}\slash B}$ and compute
\[
\begin{split}
S_b (\nabla_X\alpha,\lambda_{SW})&=\int_{C_b} \left< \nabla_X\alpha\wedge\lambda_{SW} \right> \\
&=\int_{C_b}\left< \alpha\wedge\nabla_X\lambda_{SW} \right> -d\left(\int_{C_{X(b)}}\left< \alpha\wedge\lambda_{SW}\right> \right)
\\
&=0,
\end{split}
\]
because both $\lambda_{SW}$, as well as its derivatives $\nabla_X\lambda_{SW}$ are holomorphic differentials.
Here $\left<\hdots \right>$ indicates that the pairing on $\mathfrak{t}$ has been used.
Since ${\mathcal F}^1$ is defined as the symplectic complement of $\lambda_{SW}$, 
this proves that $\nabla{\mathcal F}^2\subseteq{\mathcal F}^1$. This completes the proof of Griffiths transversality.

Finally, the CY-condition says that $\nabla{\mathcal F}^3$ should generate ${\mathcal F}^2$. 
But this is clearly implied by Proposition \ref{cy-hs}.
\end{proof}
\begin{remark}
The polarization $S$ has the wrong signature for a full VPHS of weight $3$.
Since this signature is not used in sections \ref{c5} \& \ref{c6} we can endow the base of the Hitchin system with a projective special (K\"ahler) geometry
and apply the results stated there. 
\end{remark}

\subsection{The derivative of the period map}
We give two expressions for the derivative of the period map corresponding to the family of cameral covers $f:{\mathscr{C}} \to B$. One of them (theorem \ref{dp2}) is inspired by the fact that the variation of Hodge structure of weight $1$ can be refined in a natural way to a variation of Hodge structure of weight $3$, which is reminiscent of a family of Calabi-Yau threefolds. The other expression (theorem \ref{thmdp3}) is a residue formula originally due to Balduzzi \cite{BALD:2006}, who generalized a formula of Pantev. Similar formulas are known for matrix models, see e.g. \cite{KRIC:1994}.

Given a base curve and a complex reductive group, consider the family of cameral curves $f:{\mathscr{C}} \to B$ with central fiber $C_{b_0}$.
Associated to this family is a {\emph{period map}} cf. \eqref{3.14}
\[
\Pi:B \to \check D_{lag}
\]
which is given by the embedding ${\mathcal F}^2 \subset \mathcal{V}_{b}$ composed with parallel transport using the Gauss-Manin connection. Recall from (\ref{estms}) that this is a lagrangian embedding with respect to the natural symplectic pairing on $\mathcal{V}_{b_0}$. We are interested in the derivative of the period map
\[
d\Pi_b: \left(T_{B}\right)_{b} \to Hom \left({\mathcal F}_b^2, \left({\mathcal F}_b^2 \right)^* \right)
\]
In terms of the Kodaira-Spencer map and the Gauss-Manin connection a theorem of Griffiths gives
\[
d\Pi_b(X) \left( \alpha,\beta \right) = S_b \left(  \alpha,C_X \beta \right)
\]
Using the natural isomorphism $\left(T_{B}\right)_b \cong  {\mathcal F}^2_b$ given by $X\to \nabla_X \lambda_b$ the derivative $d\Pi$ becomes a tensor on $B$:
\[
d\Pi_b: \left(T_{B}\right)_b \to \left(\left(T^*_{B}\right)_b  \right)^{\otimes 2}
\]
which is given by
\begin{equation}
\label{dp1}
d\Pi_b(X)(Y,Z) = S_b \left( \nabla_Y \lambda_b,\nabla_X \nabla_Z\lambda_b \right)
\end{equation}
Integration by parts combined with the $\nabla$-flatness of $S_{b}$ shows that (this is one of Riemann's bilinear relations)
\[
d\Pi_b: \left(T_{B}\right)_b  \to Sym^2   \left(T^*_{B}\right)_b 
\]
It is well-known \cite{DONA-MARK:1996} that integrable systems give special period maps in the sense that $d\Pi$ is a cubic
\[
d\Pi \in H^0\left( B,Sym^3 \left( T^*_{B} \right) \right)
\]
In the case of the Hitchin system, we can use the variation of weight $3$ given in the previous section together with flatness of $\nabla$ to conclude that $d\Pi(X,Y,Z)$ is indeed symmetric in its first and last arguments:
\begin{eqnarray*}
d\Pi_b (X,Y,Z) - d\Pi_b(Z,Y,X) &=& S_b \left( \nabla_Y \lambda_b,  \nabla_{[X,Z]} \lambda_b  \right)= 0
\end{eqnarray*}
We now arrive at a formula for $d\Pi$ which is reminiscent of a family of Calabi-Yau threefolds, with $\lambda_b$ playing the role of the holomorphic three-form.
\begin{theorem}
\label{dp2}
The derivative of the period map is given by (compare with \eqref{5.13})
\[
d\Pi_b(X,Y,Z) =- \int_{C_{b}} \left<  \lambda_b \wedge \nabla_X \nabla_Y \nabla_Z \lambda_b  \right>
\]
\end{theorem}
\begin{proof}
Use integration by parts with respect to $Z$ in (\ref{dp1}), the $\nabla$-flatness of $S_b$ and the symmetry in $X,Y,Z$.
\end{proof}
\begin{remark}
In \cite{DIAC-DONA-PANT:2007} a family of noncompact CY-threefolds was constructed in the case of $ADE$ groups whose variation of mixed Hodge structure of weight $3$ turns out to be pure, and in fact a Tate twist of a variation of Hodge structure of weight $1$, which is compatible with the fact that $S$ defines an indefinite polarization in weight $3$. 
The authors of \cite{DIAC-DONA-PANT:2007} have shown that the Yukawa cubic of this family of threefolds corresponds to the cubic above.
\end{remark}

The expression in theorem \ref{dp2} is not manifestly symmetric in its arguments. There is another, more symmetric,  formula due to Balduzzi \cite{BALD:2006} who generalized a result for $G=SL_2$ by Pantev. 
We will give a different derivation of his result here, which uses the \v{C}ech-de Rham complex as described in section \ref{section cdr}.
We will choose coordinates on $U\cap V$ suggested by the cameral cover $\pi:C_b \to C$: one can pull back an affine coordinate on $C$ via $\pi$ to serve as a local coordinate $z_U$ on $U$. For a generic point $b \in B$ the cover has second order branch points, which we will view as maps
\begin{displaymath}
    \xymatrix{ p:B \ar[r]^r \ar@{.>}[r] & V \ar[r]^{z_U} \ar@{.>}[r] & \mathbb{C}}
\end{displaymath}
Given the branch point $p$, a suitable holomorphic coordinate on the component of $V$ containing $r(b)$ is given by 
\[
z_V = \sqrt{z_U-p(b)}.
\]
The Seiberg-Witten differential has a second order zero at each of the branch points and can be represented in the \v{C}ech-de Rham complex by
\[
\lambda_b = \left(fz_V^2dz_V|_U,fz_V^2dz_V|_V,0\right)
\]
where $f$ is a $\mathfrak{t}$-valued holomorphic function on $V$ with $f\circ r(b)\neq 0$. 
The horizontal lifts $X_U,X_V$ of a vector field $X$ on $B$ are determined by the chosen coordinates via
\[
X_U(z_U)=0 \qquad X_V(z_V)=0
\]
We are now ready to compute the contribution to $d\Pi$ coming from the component of $V$ containing $p$.
A straightforward computation using (\ref{GM}) and the fact that
\[
\iota_{X_U}  dz_V  = \frac{-L_X(p)(b)}{2z_{V}}
\]
now gives
\[
\nabla_Y \lambda_b = \left(*,*, -\frac{z_{V}f}{2}L_Y(p)(b) \right)
\]
which has a first order zero at the branch point.
The first two terms will not contribute to (\ref{dp1}), so we omit them here. Acting with $\nabla_X$ gives
\begin{eqnarray*}
\nabla_X \nabla_Y \lambda_b &=&  \left(*,*,*-L_{X_U}\left( \frac{z_{V}f}{2}L_Y(p)(b)  \right) \right) \\
&=& 
\left(*,*,*+\frac{f}{4z_{V}}L_X(p)(b)L_Y(p)(b)\right)
\end{eqnarray*}
Only the term containing a pole at the branch point is displayed and terms which are irrelevant for (\ref{dp1}) are omitted. Using (\ref{trres}), we now arrive at the following result
\[
d\Pi(X,Y,Z) = \sum_{p\in D_{br}}  \frac{L_X(p)L_Y(p)L_Z(p)}{8}  \left\langle f\circ r, f\circ r \right\rangle
\]
where $\left\langle .,. \right\rangle$ denotes the pairing between two elements of $\mathfrak{t}$.
For semi-simple Lie groups there is only one Weyl-invariant pairing up to a scalar, which is the Killing form. 
It gives an isomorphism $\mathfrak{t}\cong \mathfrak{t}^*$ and the pairing can be expressed in terms of the root system $R$ as
\begin{equation}
\label{killing}
\left< h_1,h_2 \right> = \sum_{\alpha \in R}\left< h_1,\alpha \right> \left< \alpha,h_2\right>
\end{equation}
The quadratic residue $Res_p^2$ of a quadratic differential at a point $p$ is defined as the coefficient of $z^{-2}dz\otimes dz$ in a Laurent expansion in terms of a coordinate $z$ centered at $p$, and is independent of $z$.
\begin{theorem}[Balduzzi]
\label{thmdp3}
For semi-simple groups
\begin{equation}
\label{dp3}
d\Pi_b(X,Y,Z) =  \sum_{p\in D_{br}} \sum_{\alpha \in R}Res_{p}^2 \left[ \frac{\left< \alpha,\nabla_X \lambda_b \right>\otimes\left< \alpha,\nabla_Y \lambda_b \right> \otimes\left< \alpha,\nabla_Z \lambda_b \right>}{\left< \alpha,\lambda_b \right>} \right]
\end{equation}
\end{theorem}
\begin{proof}
The quotient of two holomorphic differentials is a meromorphic function, so the term in brackets is a meromorphic quadratic differential.
From the computation of the Gauss-Manin derivatives in the \v{C}ech-de Rham complex given above one finds that the derivatives of $\lambda$ have an expansion around the branch points in terms of $z_V$
\[
\nabla_X \lambda = \left[ L_X(p)\frac{f}{2} + \mathcal{O}(1) \right]dz_V 
\]
Similarly
\[
\frac{\left< \alpha, \nabla_X \lambda \right>}{\left<\alpha, \lambda \right>} = \frac{1}{2z_V^2}\left[ \frac{L_X(p)\left< \alpha,f \right>}{\left< \alpha, f \right>} + \mathcal{O}(1)  \right]
= \frac{1}{2z_V^2} \left[ L_X(p) + \mathcal{O}(1) \right]
\]
Taking the quadratic residue and using (\ref{killing}) directly gives the desired result.
\end{proof}
\begin{remark}
Replacing the root system by an orthonormal basis for the dual pairing gives an analogous expression for $d\Pi$ in the case of reductive non-semisimple groups.
\end{remark}

\section{The Frobenius manifold}
\label{Frob}
The results of the previous section show that the Hitchin system gives rise to projective special geometry as in section \ref{c5} on $\mathbf{P}(B)$. We illustrate in this case the choices necessary to define a Frobenius manifold structure: a natural generator $\lambda_0\in F_0^3$ is provided by the Seiberg-Witten differential, and a choice of opposite filtration $U_\bullet$ is described geometrically in terms of a choice of cycles on the cameral curve.

\subsection{The opposite filtration}
Recall the discussion of opposite filtrations in section \ref{c1.2}. There is a natural procedure to define an opposite filtration on $\mathcal{V}$, viewed as a VHS of weight $3$, as follows:
fix $b\in B$, and consider $\left(\mathcal{V}_\mathbb{Z}\right)_b^*=H_1(C_b,\Lambda)^W$. Combining the inner product on $\Lambda$ with the intersection form on $H_1(C_b,\Z)$ defines a symplectic form $\mathcal{I}$, the dual of $S$, on the lattice $H_1(C_b,\Lambda)^W$:
\[
\mathcal{I}(c_1,c_2):=\left< c_1\cdot c_2\right>,
\]
for $c_1,c_2\in H_1(C_b,\Lambda)^W$.

Now we choose a lagrangian subspace $L^2\subset H_1(C_b,\Lambda)^W$ and a one-dimensional subspace $L^3\subset L^2$ of it, subject to the condition
\begin{equation}
\label{cond-basis}
L^3 \not\subset \ker \lambda_b
\end{equation}
We will also need the complement $L^1=\left( L^3 \right)^{\perp \mathcal{I}}$ of $L^3$ with respect to $\mathcal{I}$.
With this we define 
\[
\begin{split}
(U_0)_b&:=\{v\in\mathcal{V}_b,~L^1\subset \ker v \}\\
(U_1)_b&:=\{v\in\mathcal{V}_b,~L^2\subset \ker v\}\\
(U_2)_b&:=\{v\in\mathcal{V}_b,~L^3\subset \ker v\}\\
(U_3)_b&:=\mathcal{V}_b
\end{split}
\]
We extend these subspaces by parallel transport to $\nabla$-flat subbundles $U_\bullet$ in a small 
neighbourhood of $b\in B$.
\begin{proposition}
\label{propopp}
$U_\bullet$ defines an opposite filtration for the variation of Hodge-like filtrations on $\mathcal{V}$.
\end{proposition}
\begin{proof}
By construction, the subbundles $U_\bullet$ are $\nabla$-flat. Next, let us check that
\[
\mathcal{V}_b={F}^p_b \oplus \left(U_{p-1}\right)_b,
\]
for all $p=0,\ldots,3$. For this, first observe that $\rk(\mathcal{V})=\rk({F}^p)+\rk (U_{p-1})$,
 so we just have to verify that ${F}^p_b\cap(U_{p-1})_b=\{0\}$.
 Since ${F}^3_b$ is spanned by $\lambda_b$, this follows for $p=3$ from condition
\eqref{cond-basis}. For $p=2$, we have that $\alpha\in {F}^2\cap U_1$ implies that
$\alpha\in H^0(C_b,\mathfrak{t}\otimes K_{C_b})^W$, i.e., $\alpha$ is a holomorphic differential, and $L^2\subset \ker \alpha$.
By Abel's theorem, this implies that $\alpha$ has to be zero.
An element $\alpha \in {F}^1 \cap U_0$ satisfies by definition
\[
\alpha \in \left( {F}^3 \right)^{\perp S} \qquad \& \qquad \left( L^3 \right)^{\perp \mathcal{I}} \subset \ker \alpha
\]
But the fact that $S$ and $\mathcal{I}$ are dual implies that \eqref{cond-basis} is equivalent to
\[
\left( L^3 \right)^{\perp \mathcal{I}} \not\subset \ker \left( \lambda_b^{\perp S} \right)
\]
Finally, the condition that $S(U_p,U_{2-p})=0$ follows from the fact that $\mathcal{I}$ is the dual of $S$ and $\mathcal{I}\left( L^p,L^{4-p}\right)=0$.
\end{proof}

\subsection{Special coordinates and the prepotential}
Now we choose a symplectic basis $(a_1,\dots,a_n,b_1,\dots,b_n)$ of $H_1(C_b,\Lambda)^W$ with the sets of cycles $\{a_1\}, \{a_1,\dots,a_n\},\{a_1,\dots,a_n,b_2,\dots,b_n\}$ providing bases for $L^3,L^2,L^1$ respectively.
An alternative proof of proposition \ref{propopp} can be given by using this basis and the fact that
\[
S(\alpha,\beta) = \sum_{k=1}^{n+1} \left( \int_{a_k}\alpha\int_{b_k}\beta - \int_{b_k}\alpha\int_{a_k}\beta \right)
\]

The choice of cycles gives rise to the special coordinates
\[
z_i = \oint_{a_i} \lambda
\]
and adjoint coordinates
\[
\frac{\partial \Psi}{\partial z_i} = \oint_{b_i} \lambda
\]
on the Hitchin base $B$, where $\Psi \in \mathcal{O}_B$ denotes the prepotential. The choice of $U_0$ (or $L^3$) determines the hyperplane
\[
B_0 := \{ z_1=1 \} \subset B 
\]
and the coordinates $t_i=z_i|_{z_1=1}$ on it. 
The germ of a Frobenius manifold structure on $\mathbb{C}\times B \times \mathbb{C}^{n+1}$ is completely specified by coordinates $t_1,\dots t_{2n+2}$ which include the coordinates on $B$ just defined, together with the potential $\Phi(t)$ in \eqref{2.19} and the Euler vector field $E$ in \eqref{2.22}.